\documentclass[a4paper,reqno,12pt]{amsart}
\usepackage[margin=1.1in]{geometry}
\usepackage{kpfonts}
\usepackage{indentfirst}
\usepackage{amsfonts}
\usepackage{amscd}
\usepackage{amsthm}
\usepackage{amssymb}
\usepackage{amsmath}
\usepackage[all,cmtip]{xy}
\usepackage{tikz}
\usepackage{tikz-cd}
\usepackage[disable]{todonotes}
\usepackage{extpfeil}
\usetikzlibrary{decorations.markings,shapes.geometric}
\usepackage{enumitem}
\usepackage{setspace}
\usepackage{epsfig}
\usepackage[hidelinks]{hyperref}
\usetikzlibrary{matrix,arrows,positioning}
\usepackage{upgreek}
\usepackage[final]{microtype}
\def\MR#1{} 

\usepackage[font={small,sl}]{caption}

\def\itemNum$#1${\item $\displaystyle#1$
   \hfill\refstepcounter{equation}(\theequation)}

\makeatletter
\providecommand\@dotsep{5}
\renewcommand{\listoftodos}[1][\@todonotes@todolistname]{%
  \@starttoc{tdo}{#1}}
\makeatother

\newtheorem{Lem}{Lemma}[section]
\newtheorem{Prop}[Lem]{Proposition}

\newtheorem{Def}[Lem]{Definition}

\theoremstyle{plain}
\newtheorem{Thm}[Lem]{Theorem}

\newtheorem{Cor}[Lem]{Corollary}

\theoremstyle{plain}
\newtheorem{thmIntro}{Theorem}

\theoremstyle{definition}
\newtheorem{Rem}[Lem]{Remark}
\newtheorem{Rems}[Lem]{Remarks}

\theoremstyle{definition}
\newtheorem{examplex}[Lem]{Example}
\newenvironment{Ex}
  {\pushQED{\qed}\examplex}
  {\popQED\endexamplex}

\newlist{todolist}{itemize}{2}
\setlist[todolist]{label=$\square$}

\newcommand{\C}{\mathbb{C}}
\newcommand{\R}{\mathbb{R}}
\newcommand{\quat}{\mathbb{H}}

\newcommand{\Hom}{\textup{\text{Hom}}}

\newcommand{\Rep}{\textup{\text{Rep}}}
\newcommand{\RRep}{\textup{\text{RRep}}}
\newcommand{\Irr}{\textup{\text{Irr}}}
\newcommand{\Bord}{\textup{\text{Bord}}}
\newcommand{\Brane}{\mathcal{B}}

\newcommand{\cat}{\mathcal{C}}
\newcommand{\Vect}{\textup{\text{Vect}}}

\newcommand{\id}{\textup{\text{id}}}
\newcommand{\refl}{\textup{\text{refl}}}

\newcommand{\Ctwo}{C_2}
\newcommand{\op}{\textup{\text{op}}}
\newcommand{\ev}{\textup{\text{ev}}}
\newcommand{\Aut}{\textup{\text{Aut}}}

\newcommand{\tr}{\textup{\text{tr}}}
\newcommand{\Gh}{\hat{G}}
\newcommand{\G}{G}
\newcommand{\Pin}{\textup{\text{Pin}}}
\newcommand{\git}{/\!\!/}
\newcommand{\bb}{\tau}
\newcommand{\TFT}{\mathcal{Z}}
\newcommand{\ori}{\textup{\text{or}}}

\makeatletter
\newbox\xrat@below
\newbox\xrat@above
\newcommand{\xrightarrowtail}[2][]{%
  \setbox\xrat@below=\hbox{\ensuremath{\scriptstyle #1}}%
  \setbox\xrat@above=\hbox{\ensuremath{\scriptstyle #2}}%
  \pgfmathsetlengthmacro{\xrat@len}{max(\wd\xrat@below,\wd\xrat@above)+.6em}%
  \mathrel{\tikz [>->,baseline=-.75ex]
                 \draw (0,0) -- node[below=-2pt] {\box\xrat@below}
                                node[above=-2pt] {\box\xrat@above}
                       (\xrat@len,0) ;}}
\makeatother

\begin{document}
\title[Real Frobenius--Schur indicators]{Frobenius--Schur indicators for twisted Real representation theory and two dimensional unoriented topological field theory}

\author[L. Gagnon-Ririe]{Levi Gagnon-Ririe}
\address{Department of Mathematics and Statistics \\ Utah State University\\
Logan, Utah 84322 \\ USA}
\email{gagnon.levi@gmail.com}

\author[M.\,B. Young]{Matthew B. Young}
\address{Department of Mathematics and Statistics \\ Utah State University\\
Logan, Utah 84322 \\ USA}
\email{matthew.young@usu.edu}

\date{\today}


\begin{abstract}
We construct a two dimensional unoriented open/closed topological field theory from a finite graded group $\pi:\Gh \twoheadrightarrow \{1,-1\}$, a $\pi$-twisted $2$-cocycle $\hat{\theta}$ on $B \Gh$ and a character $\lambda: \Gh \rightarrow U(1)$. The underlying oriented theory is a twisted Dijkgraaf--Witten theory. The construction is based in the $(\Gh, \hat{\theta},\lambda)$-twisted Real representation theory of $\ker \pi$. In particular, twisted Real representations are boundary conditions and the generalized Frobenius--Schur element is its crosscap state.
\end{abstract}

\maketitle

\setcounter{tocdepth}{1}


\setcounter{footnote}{0}

\section*{Introduction}
\addtocontents{toc}{\protect\setcounter{tocdepth}{1}}
\label{sec:Intro}

Associated to a finite group $\G$ and a $U(1)$-valued $2$-cocycle $\theta$ on its classifying space $B\G$ is a two dimensional topological gauge theory known as Dijkgraaf--Witten theory \cite{dijkgraaf1990}. This is an oriented open/closed topological quantum field theory (TFT) $\TFT_{(\G,\theta)}$ with boundary conditions the category $\Rep^{\theta}(\G)$ of finite dimensional $\theta$-twisted complex representations of $\G$ \cite{freed1994,moore2006}. In particular, $\TFT_{(\G,\theta)}$ assigns a partition function to each compact oriented $2$-manifold with boundary components labelled by twisted representations. Open/closed TFT was introduced as a framework to axiomatize the structure of topological D-branes in string theory \cite{lazaroiu2001,kapustin2004,moore2006} and has found a variety of applications in pure mathematics \cite{costello2007,blumberg2009,abouzaid2010}. The open/closed structure of Dijkgraaf--Witten theory plays an important role in the descriptions of D-branes in orbifold string theory \cite{dijkgraaf1990}, generalized symmetries in quantum field theory \cite{sharpe2015,huang2021} and boundary degrees of freedom in topological phases of matter \cite{shiozaki2017}.

Open/closed TFTs on unoriented---and possibly non-orientable---manifolds play a central role in orientifold string theory \cite{hori2008} and related mathematics \cite{mbyoung2020,freed2021,georgieva2021,noohiYoung2022}. In condensed matter physics, unoriented TFTs in general, and Dijkgraaf--Witten theory in particular, model topological phases of matter with time reversal symmetry \cite{freed2013b,kapustin2017,barkeshli2020}. The main result of this paper is an algebraic construction of a class of unoriented lifts of the oriented open/closed Dijkgraaf--Witten theories $\TFT_{(\G,\theta)}$. 

\begin{thmIntro}[{Theorem \ref{thm:twistRealRepTFT}}]
\label{thm:twistRealRepTFTIntro}
A triple $(\Gh, \hat{\theta}, \lambda)$ consisting of a short exact sequence of finite groups
\[
1 \rightarrow \G \rightarrow \Gh \xrightarrow[]{\pi} \Ctwo = \{1,-1\} \rightarrow 1,
\]
a $\pi$-twisted $2$-cocycle $\hat{\theta}$ on $B\Gh$ which restricts to $\theta$ on $B \G$ and a character $\lambda: \Gh \rightarrow U(1)$ defines a two dimensional unoriented open/closed topological field theory $\TFT_{(\Gh,\hat{\theta},\lambda)}$ whose oriented sector is a subtheory of $\TFT_{(\G,\theta)}$.
\end{thmIntro}

A number of authors have studied $\TFT_{(\Gh,\hat{\theta},\lambda)}$ under the assumption that $\Gh = \G \times \Ctwo$ is the trivial extension, $\hat{\theta}$ is in the image of the map $H^2(B\G ;\Ctwo) \rightarrow H^{2}(B\Gh;U(1)_{\pi})$ and $\lambda$ is trivial \cite{karimipour1997,alexeevski2006,turaev2007,loktev2011,snyder2017}. For general $(\Gh, \hat{\theta})$ and trivial $\lambda$, a topological construction of the closed sector of $\TFT_{(\Gh,\hat{\theta},1)}$, and its higher dimensional analogues, was given in \cite{mbyoung2020} while a $\G$-equivariant extension of the closed sector of $\TFT_{(\Gh,\hat{\theta},1)}$ was given in \cite{kapustin2017}. We emphasize that for the applications of unoriented Dijkgraaf--Witten theory mentioned before Theorem \ref{thm:twistRealRepTFTIntro}, general input data $(\Gh,\hat{\theta},\lambda)$ is required; see Remark \ref{rem:TypeIITwists}. As explained below, general input data is also natural from the representation theoretic and $K$-theoretic perspectives.

Theorem \ref{thm:twistRealRepTFTIntro} is proved using an algebraic characterization of unoriented TFTs, Theorem \ref{thm:genStruAlg}, which builds off characterizations of oriented closed and open/closed TFTs \cite{dijkgraaf1989,abrams1996,lazaroiu2001,moore2006,alexeevski2006,lauda2008}, unoriented closed TFTs \cite{turaev2006} and unoriented open/closed TFTs with a single boundary condition \cite{alexeevski2006}. The algebraic data required to define an unoriented open/closed TFT includes:
\begin{itemize}
\item A commutative Frobenius algebra $A$; this defines the oriented closed sector.
\item A Calabi--Yau category $\Brane$; this defines the oriented open sector.
\item An isometric involution $p: A \rightarrow A$ and a \emph{crosscap state} $Q \in A$, the latter corresponding to the value of the TFT on the compact M\"{o}bius strip; this defines the unoriented closed sector.
\item A strict contravariant involution of $\Brane$, that is, a functor $P: \Brane^{\op} \rightarrow \Brane$ which squares to the identity, which is moreover required to be the identity on objects; this defines the unoriented open sector.
\end{itemize}
The data (and that which we have omitted here) is required to satisfy a number of coherence conditions. The oriented theory $\TFT_{(\G,\theta)}$ is defined by the commutative Frobenius algebra $HH_{\bullet}(\Rep^{\theta}(\G)) \simeq Z(\C^{\theta^{-1}}[\G])$ with the Haar bilinear form $\langle-, -\rangle_{\G}$ and Calabi--Yau category $\Rep^{\theta}(\G)$. Motivated by the search for the data required to define the unoriented lift $\TFT_{(\Gh,\hat{\theta},\lambda)}$, in Section \ref{sec:FSInd} we construct and study a contravariant involution $(P^{(\Gh,\hat{\theta},\lambda)},\Theta^{(\Gh,\hat{\theta},\lambda)})$ of $\Rep^{\theta}(\G)$. The functor $P^{(\Gh,\hat{\theta},\lambda)}$ acts non-trivially on objects and so is not an admissible choice for the defining data of $\TFT_{(\Gh,\hat{\theta},\lambda)}$. A key representation theoretic observation is that the homotopy fixed points of $(P^{(\Gh,\hat{\theta},\lambda)},\Theta^{(\Gh,\hat{\theta},\lambda)})$ is the category of $(\Gh,\hat{\theta},\lambda)$-twisted Real representations of $\G$. The Real representation theory of $\G$ was originally studied by Wigner \cite{wigner1959} and Dyson \cite{dyson1962} as a generalization of real and quaternionic representation theory in the context of anti-unitary symmetries in quantum mechanics. More recently, Real representation theory has been developed from the related perspective of twisted equivariant $KR$-theory \cite{atiyah1969,karoubi1970,freed2013b,noohiYoung2022} and categorical representation theory \cite{mbyoung2021b,rumyninYoung2021,rumynin2021b}. In the $K$-theoretic setting, general pairs $(\hat{\theta}, \lambda)$ are required to realize all $KR$-theory twists. Motivated by the above perspectives, we consider the element
\[
\nu_{(\Gh,\hat{\theta},\lambda)}
=
\sum_{\varsigma \in \Gh \setminus \G} \frac{\lambda(\varsigma)}{\hat{\theta}([\varsigma \vert \varsigma])} l_{\varsigma^2} \in HH_{\bullet}(\Rep^{\theta}(\G)).
\]
The role of $\nu_{(\Gh,\hat{\theta},\lambda)}$ in Real representation theory is summarized by the next result.

\begin{thmIntro}[{Theorem \ref{thm:LefGroupAlg} and Corollary \ref{cor:FSInd}}]
\label{thm:LefGroupAlgIntro}
Let $V$ be a $\theta$-twisted representation of $\G$ with character $\chi_V$. Then $\langle \chi_V, \nu_{(\Gh,\hat{\theta},\lambda)} \rangle_{\G}$ is equal to the trace of the involution
\[
\Hom_{\G}(V,P^{(\Gh,\hat{\theta},\lambda)}(V)) \rightarrow \Hom_{\G}(V,P^{(\Gh,\hat{\theta},\lambda)}(V)).
\qquad
f \mapsto P^{(\Gh,\hat{\theta},\lambda)}(f) \circ \Theta^{(\Gh,\hat{\theta},\lambda)}_V.
\]
In particular, if $V$ is irreducible, then
\[
\langle \chi_V, \nu_{(\Gh,\hat{\theta},\lambda)} \rangle_{\G}
=
\begin{cases}
1 & \mbox{if and only if $V$ lifts to a $(\Gh,\hat{\theta},\lambda)$-twisted Real representation},\\
-1 & \mbox{if and only if $V$ lifts to a $(\Gh,\delta\hat{\theta},\lambda)$-twisted Real representation}, \\
0 & \mbox{otherwise},
\end{cases}
\]
where $\delta$ is a representative of the generator of $H^{2}(B \Ctwo;U(1)_{\pi}) \simeq \Ctwo$.
\end{thmIntro}

The element $\nu_{(\Gh,\hat{\theta},\lambda)}$ recovers under various specializations of the data $(\Gh,\hat{\theta},\lambda)$ other generalized Frobenius--Schur elements \cite{frobenius1906,gow1979,turaev2007,ichikawa2023}. In particular, the second statement in Theorem \ref{thm:LefGroupAlgIntro} shows that $\nu_{(\Gh,\hat{\theta},\lambda)}$ is a generalization to twisted Real representation theory of the classical Frobenius--Schur element. Theorem \ref{thm:LefGroupAlgIntro} and a complete understanding of the $\theta$-twisted representation theory of $\G$ suffices to understand the $(\Gh,\hat{\theta},\lambda)$-twisted Real representation theory of $\G$.

Returning to the proof of Theorem \ref{thm:twistRealRepTFTIntro}, we take for $\Brane$ the Calabi--Yau category of $(\Gh,\hat{\theta},\lambda)$-twisted Real representations of $\G$ and their $\G$-equivariant linear maps. We view this as an orientifold-type construction, with $\Rep^{\theta}(\G)$ seen as the category of $D$-branes in an oriented string theory and $\Brane$ the category of $D$-branes which survive the orientifold projection defined by $(P^{(\Gh,\hat{\theta},\lambda)},\Theta^{(\Gh,\hat{\theta},\lambda)})$. The category twisted Real representations is a non-full subcategory of $\Brane$ and the forgetful functor $\Brane \rightarrow \Rep^{\theta}(\G)$ respects Calabi--Yau structures. Moreover, $\Brane$ inherits a contravariant involution which is the identity on objects and $A= HH_{\bullet}(\Rep^{\theta}(\G))$ inherits an isometric involution $p$. We take for the crosscap state $Q$ the generalized Frobenius--Schur element $\nu_{(\Gh,\hat{\theta},\lambda)}$. It remains to verify the coherence conditions. A mild generalization of the first equality in Theorem \ref{thm:LefGroupAlgIntro} (proved in Theorem \ref{thm:LefGroupAlg}) is the unoriented counterpart of the famous Cardy condition, asserting the equality of two ways of evaluating a M\"{o}bius strip diagram with boundary condition $V$. The remaining coherence conditions required of the crosscap state, involution $p$ and boundary-bulk and bulk-boundary maps are verified using the calculus of twisted cocycles. In Section \ref{sec:partFun}, we compute partition functions of $\TFT_{(\Gh,\hat{\theta},\lambda)}$.

\subsection*{Acknowledgements}
The work of M. B. Y. was supported by National Science Foundation grant DMS-2302363 and a Simons Foundation Collaboration Grant for Mathematicians (Award ID 853541).

\section{Background material}
\label{sec:background}

Throughout the paper the ground field is $\C$ and vector spaces are finite dimensional. Linear duality is $(-)^{\vee} = \Hom_{\C}(-,\C)$. Denote by $U(1)$ the group of unit norm complex numbers and $C_n$ the cyclic group of order $n$, seen as a multiplicative group.

\subsection{Group cohomology}

Let $K$ be a finite group and $M$ a left $K$-module. We regard the underlying abelian group of $M$ as multiplicative. Let $C^{\bullet}(BK; M)$ be the complex of normalized simplicial cochains on $BK$ with coefficients in $M$. An element $\theta \in C^n(BK; M)$ is a function
\[
\theta: K^n \rightarrow M,
\qquad
(k_n, \dots, k_1) \mapsto \theta([k_n \vert \cdots \vert k_1])
\]
whose value is the identity if any $k_i$ is the identity. The differential $d \theta$ of an $(n-1)$-cochain $\theta$ is defined so that $d \theta ([k_n \vert \cdots \vert k_1])$ is equal to
\[
k_n \cdot \theta([k_{n-1} \vert \cdots \vert k_1])
\prod_{j=1}^{n-1} \theta([k_n \vert \cdots \vert k_{j+1} k_j\vert \cdots \vert k_1])^{(-1)^{n-j}}
\times \theta([k_n \vert \cdots \vert k_2])^{(-1)^n}.
\]
Write $Z^{\bullet}(BK;M)$ and $H^{\bullet}(BK;M)$ for the cocycles and cohomologies of $C^{\bullet}(BK; M)$.

When $M = U(1)$ with trivial $K$-action, write $C^{\bullet}(BK)$ for $C^{\bullet}(BK;M)$. When $\pi: \Gh \rightarrow \Ctwo$ is a group homomorphism and $M=U(1)$ with $\Gh$-action $\omega \cdot z = z^{\pi(\omega)}$, write $C^{\bullet+\pi}(B\Gh)$ for $C^{\bullet}(B\Gh;M)$. If $\pi: \Ctwo \rightarrow \Ctwo$ is the identity map, then $H^{2+\pi}(B \Ctwo) \simeq \Ctwo$; a cocycle representative $\delta$ for the generator is given by
\[
\delta ([\varsigma_2 \vert \varsigma_1])
=
\begin{cases}
-1 & \mbox{if } \varsigma_1=\varsigma_2=-1,\\
0 & \mbox{otherwise}.
\end{cases}
\]
We use the same notation for $\delta$ and its image under $\pi^*: Z^{2+\pi}(B \Ctwo) \rightarrow Z^{2+\pi}(B \Gh)$.


\begin{Lem}
Let $\pi: \Gh \rightarrow \Ctwo$ be a $\Ctwo$-graded finite group and $\hat{\theta} \in Z^{2+\pi}(B \Gh)$. For all $g_i \in \G$, $\omega \in \Gh$ and $\varsigma \in \Gh \setminus \G$, the following equalities hold:
\begin{equation}
\label{eq:2cocycleKey}
\frac{\hat{\theta}([\omega g_2 \omega^{-1} \vert \omega g_1 \omega^{-1}])}{\hat{\theta}([g_2 \vert g_1])^{\pi(\omega)}}
=
\frac{\hat{\theta}([\omega g_2 \omega^{-1} \vert \omega])}{\hat{\theta}([\omega \vert g_2])} \frac{\hat{\theta}([\omega g_1 \omega^{-1} \vert \omega])}{\hat{\theta}([\omega \vert g_1])}
\left( \frac{\hat{\theta}([\omega g_2 g_1 \omega^{-1} \vert \omega])}{\hat{\theta}([\omega \vert g_2 g_1])} \right)^{-1}
\end{equation}
\begin{equation}
\label{eq:oddConj}
\frac{\hat{\theta}([\omega \varsigma \omega^{-1} \vert \omega \varsigma \omega^{-1}])}{\hat{\theta}([\varsigma \vert \varsigma])^{-\pi(\omega)}}
=
\frac{\hat{\theta}([\omega \vert \varsigma^2])}{\hat{\theta}([\omega \varsigma^2 \omega^{-1} \vert \omega])}.
\end{equation}
\end{Lem}

\begin{proof}
Both equalities follow from repeated use of the $2$-cocycle condition on $\hat{\theta}$.
\end{proof}

\subsection{Twisted representation theory}
\label{sec:twistRepThy}
We recall background on twisted representation theory following \cite{karpilovsky1985}. Let $\G$ be a finite group and $\theta \in Z^2(B\G)$.

\begin{Def} A \emph{$\theta$-twisted} (or \emph{$\theta$-projective}) \emph{representation of $\G$} is pair $(V,\rho)$ consisting of a vector space $V$ and a map $\rho: \G \rightarrow GL(V)$ which satisfies $\rho(e)=\id_{V}$ and
\[
\rho(g_2) \circ \rho(g_1) = \theta([g_2 \vert g_1]) \rho(g_2 g_1),
\qquad
g_1, g_2 \in G.
\]
\end{Def}

We often write $V$ or $\rho_V$ for $(V,\rho)$. The category $\Rep^{\theta}(\G)$ of $\theta$-twisted representations and their $\G$-equivariant linear maps is $\C$-linear finite semisimple.

The $\theta$-twisted group algebra $\C^{\theta}[G]$ is the $\C$-algebra with basis $\{l_g \mid g \in \G\}$ and multiplication $l_{g_2} \cdot l_{g_1} = \theta([g_2 \vert g_1]) l_{g_2 g_1}$. The category of finite dimensional $\C^{\theta}[G]$-modules is equivalent to $\Rep^{\theta}(\G)$. We sometimes interpret $\C^{\theta}[G]$ as functions on $\G$, in which case $l_g$ the $\delta$-function at $g$.
The centre $Z(\C^{\theta}[\G])$ consists of elements $\sum_{g \in \G} a_g l_g$ whose coefficients satisfy
\[
a_{hgh^{-1}}
=
\uptau(\theta)([h]g)^{-1} a_g,
\qquad
g,h \in \G.
\]
Here $\uptau(\theta)([h]g) = \frac{\theta([h g h^{-1} \vert h])}{\theta([h \vert g])}$ are the components of a $1$-cocycle $\uptau(\theta)$ on the loop groupoid of $B \G$ called the \emph{loop transgression} of $\theta$ \cite[Theorem 3]{willerton2008}. Define a non-degenerate symmetric bilinear form on $\C^{\theta}[\G]$ by
\[
\langle \sum_{g \in \G} a_g l_g, \sum_{h \in \G} b_h l_h \rangle_{\G,\theta} = \frac{1}{\vert G \vert} \sum_{g \in \G} \theta([g^{-1} \vert g]) a_{g^{-1}} b_g.
\]

The character of $(V,\rho) \in \Rep^{\theta}(\G)$ is the function $\chi_V: \G \rightarrow \C$, $g \mapsto \tr_V\, \rho(g)$. A short calculation shows that $\chi_V(hgh^{-1}) = \uptau(\theta)([h]g) \chi_V(g)$. Functions $\G \rightarrow \C$ with this conjugation equivariance are elements of $Z(\C^{\theta^{-1}}[\G])$ and are called $\theta$-twisted class functions. Characters of irreducible $\theta$-twisted representations form an orthonormal basis of $Z(\C^{\theta^{-1}}[\G])$ with respect to $\langle-,-\rangle_{\G}:=\langle -, -\rangle_{\G,\theta^{-1}}$.

Given $(V,\rho_V) \in \Rep^{\theta}(\G)$, define $(V^{\vee}, \rho_{V^{\vee}}) \in \Rep^{\theta^{-1}}(\G)$ by $\rho_{V^{\vee}}(g) = (\rho_V(g)^{-1})^{\vee}$. For ease of notation, we write $\rho_V(g)^{- \vee}$ for $(\rho_V(g)^{-1})^{\vee}$.

\subsection{Categories with duality}

\begin{Def}
\label{def:catWDual}
\begin{enumerate}
\item A \emph{category with duality} is a triple $(\cat,P,\Theta)$ consisting of a category $\cat$, a functor $P: \cat^{\op} \rightarrow \cat$ and a natural isomorphism $\Theta: \id_{\cat} \Rightarrow P \circ P^{\op}$ whose components satisfy
\begin{equation}
\label{eq:catWDualCoher}
P(\Theta_V) \circ \Theta_{P(V)} = \id_{P(V)},
\qquad
V \in \cat.
\end{equation}
The duality structure $(P,\Theta)$ is \emph{strict} if $\Theta$ is the identity natural transformation.

\item A \emph{homotopy fixed point} of $(\cat,P,\Theta)$ is a pair $(V,\psi_V)$ consisting of an object $V \in \cat$ and an isomorphism $\psi_V: V \rightarrow P(V)$ which satisfies $P(\psi_V) \circ \Theta_V = \psi_V$.
\end{enumerate}
\end{Def}


We interpret $(P,\Theta)$ as defining a categorical $\Ctwo$-action on $\cat$ in which the generator acts contravariantly. Motivated by this, let $\cat^{h\Ctwo}$, $\cat^{\tilde{h}\Ctwo}$ be the categories with objects homotopy fixed points and morphisms
\[
\Hom_{\cat^{h\Ctwo}}((V,\psi_V),(W,\psi_W)) = \{\phi \in \Hom_{\cat}(V,W) \mid \psi_V = P(\phi) \circ \psi_W \circ \phi \},
\]
\[
\Hom_{\cat^{\tilde{h}\Ctwo}}((V,\psi_V),(W,\psi_W)) = \Hom_{\cat}(V,W).
\]
Let $P^{\tilde{h}\Ctwo}: (\cat^{\tilde{h}\Ctwo})^{\op} \rightarrow \cat^{\tilde{h}\Ctwo}$ be the identity on objects and send a morphism $\phi: (V,\psi_V) \rightarrow (W,\psi_W)$ to $P^{\tilde{h}\Ctwo}(\phi) = \psi_V^{-1} \circ P(\phi) \circ \psi_W$. Let $\Theta^{\tilde{h}\Ctwo}: \id_{\cat^{\tilde{h}\Ctwo}} \Rightarrow P^{\tilde{h}\Ctwo} \circ (P^{\tilde{h}\Ctwo})^{\op}$ be the identity natural transformation.

\begin{Lem}
\label{lem:homoFixDual}
The triple $(\cat^{\tilde{h}\Ctwo},P^{\tilde{h}\Ctwo},\Theta^{\tilde{h}\Ctwo})$ is a category with strict duality. Moreover, $P^{\tilde{h}\Ctwo}$ is the identity on objects.
\end{Lem}

\section{A Frobenius--Schur indicator for twisted Real representation theory}
\label{sec:FSInd}

\subsection{Twisted Real representation theory}
\label{sec:RealProjRep}

The Real representation theory of a finite group has been studied by many authors as a generalization of representation theory over $\R$ or $\quat$ \cite{wigner1959,dyson1962,atiyah1969,karoubi1970,freed2013b,mbyoung2021b}. We establish relevant aspects of the twisted form of this theory following \cite[\S 3.2]{mbyoung2021b}.

Let $\pi: \Gh \rightarrow \Ctwo$ be a $\Ctwo$-graded finite group with $\pi$ surjective. Fix $\hat{\theta} \in Z^{2+\pi}(B \Gh)$ and a character $\lambda: \Gh \rightarrow U(1)$. Note that $\lambda$ can be interpreted as an element of $Z^1(B\Gh)$. Denote by $\G = \ker \pi$ and $\theta \in Z^2(B\G)$ the restriction of $\hat{\theta}$ along $B\G \rightarrow B \Gh$.

An element $\varsigma \in \Gh \backslash \G$ determines a $\C$-linear exact duality structure $(P^{(\hat{\theta},\lambda,\varsigma)}, \Theta^{(\hat{\theta},\lambda,\varsigma)})$ on $\Rep^{\theta}(\G)$. 
On objects, we have $P^{(\hat{\theta},\lambda,\varsigma)}(V,\rho) = (V^{\vee},\rho^{(\hat{\theta},\lambda,\varsigma)})$, where
\[
\rho^{(\hat{\theta},\lambda,\varsigma)}(g) = \frac{\lambda(g)}{\uptau_{\pi}^{\refl}(\hat{\theta})([\varsigma] g)} \rho(\varsigma g^{-1} \varsigma^{-1})^{\vee}, \qquad g \in \G.
\]
The coefficients
\begin{equation*}
\label{eq:twistTrans2Cocyc}
\uptau^{\refl}_{\pi}(\hat{\theta})([\omega]g) = \hat{\theta}([g^{-1} \vert g])^{\frac{\pi(\omega)-1}{2}} \frac{\hat{\theta}([\omega g^{\pi(\omega)} \omega^{-1} \vert \omega])}{\hat{\theta}([\omega \vert g^{\pi(\omega)}])},
\qquad
g \in \G, \; \omega \in \Gh
\end{equation*}
are best understood in terms of orientation-twisted loop transgression \cite[Theorem 2.8]{noohiYoung2022}, which is a cochain map
\[
\uptau^{\refl}_{\pi} : C^{\bullet+\pi}(B \Gh) \rightarrow C^{\bullet-1}(B (\G \git_R \Gh)).
\]
The codomain is simplicial cochains on the classifying space of the quotient groupoid $\G \git_R \Gh$ resulting from the Real conjugation action of $\Gh$ on $\G$: $\omega \cdot g = \omega g^{\pi(\omega)} \omega^{-1}$, $\omega \in \Gh, \; g \in G$. In geometric terms, $\G \git_R \Gh$ is the unoriented loop groupoid of $B \Gh$, that is, the quotient of the loop groupoid of $B\G$ by the $\Ctwo$-action which reverses orientation of loops and acts on $B\G$ by deck transformations. Continuing, on morphisms $P^{(\hat{\theta},\lambda,\varsigma)}$ is $\C$-linear duality. The natural isomorphism $\Theta^{(\hat{\theta},\lambda,\varsigma)}$ is defined by its components
\[
\Theta^{(\hat{\theta},\lambda,\varsigma)}_V
=
\frac{\lambda(\varsigma)}{\hat{\theta}([\varsigma \vert \varsigma])} \ev_V \circ \rho(\varsigma^{2})^{-1},
\qquad
(V,\rho) \in \Rep^{\theta}(\G)
\]
where $\ev_V : V \rightarrow V^{\vee \vee}$ is the evaluation isomorphism of underlying vector spaces. The normalization of $\Theta^{(\hat{\theta},\lambda,\varsigma)}_V$ ensures that the coherence condition \eqref{eq:catWDualCoher} holds.

\begin{Def}
The \emph{category $\RRep^{(\hat{\theta},\lambda)}(\G)$ of $(\hat{\theta},\lambda)$-twisted Real representations of $\G$} is the homotopy fixed point category $\Rep^{\theta}(\G)^{h\Ctwo}$ of $(P^{(\hat{\theta},\lambda,\varsigma)}, \Theta^{(\hat{\theta},\lambda,\varsigma)})$.
\end{Def}

Up to equivalence, $(P^{(\hat{\theta},\lambda,\varsigma)}, \Theta^{(\hat{\theta},\lambda,\varsigma)})$ depends only on $(\Gh,[\hat{\theta}],\lambda)$. The same is therefore true of $\RRep^{(\hat{\theta},\lambda)}(\G)$ and we drop $\varsigma$ from the notation if it is fixed or the particular realization of the duality structure is not important.

Concretely, an object $(V,\psi_V) \in \RRep^{(\hat{\theta},\lambda)}(\G)$ is an isomorphism $\psi_V: V \rightarrow P^{(\hat{\theta},\lambda)}(V)$ in $\Rep^{\theta}(\G)$ which satisfies $P^{(\hat{\theta},\lambda)}(\psi_V) \circ \Theta^{(\hat{\theta},\lambda)}_V = \psi_V$. A morphism $\phi: (V,\psi_V) \rightarrow (W,\psi_W)$ in $\RRep^{(\hat{\theta},\lambda)}(\G)$ is a morphism in $\Rep^{\theta}(\G)$ which satisfies $P^{(\hat{\theta},\lambda)}(\phi) \circ \psi_W \circ \phi = \psi_V$. Note that $\phi$ is necessarily injective and $\RRep^{(\hat{\theta},\lambda)}(\Gh)$ is neither linear nor abelian.

A more standard representation theoretic interpretation of $\RRep^{(\hat{\theta},\lambda)}(\G)$ is as follows. Given a vector space $V$ and sign $\epsilon \in \Ctwo$, introduce the notation
\[
{^{\epsilon}}V
=
\begin{cases}
V & \mbox{if } \epsilon=1, \\
V^{\vee} & \mbox{if } \epsilon=-1
\end{cases}
\]
with similar notation for linear maps. A $(\hat{\theta},\lambda)$-twisted Real representation of $\G$ is then a vector space $V$ together with linear maps $\rho(\omega): \prescript{\pi(\omega)}{}{V} \rightarrow V$, $\omega \in \Gh$, which satisfy $\rho(e) = \id_V$ and
\begin{equation}
\label{eq:explicitRealRep}
\rho(\omega_2) \circ \prescript{\pi(\omega_2)}{}{\rho(\omega_1)}^{\pi(\omega_2)} \circ \ev_V^{\delta_{\pi(\omega_1), \pi(\omega_2), -1}}
=
\lambda(\omega_1)^{\frac{\pi(\omega_2)-1}{2}}\hat{\theta}([\omega_2 \vert \omega_1]) \rho(\omega_2 \omega_1).
\end{equation}
The notation $\ev_V^{\delta_{\pi(\omega_1), \pi(\omega_2), -1}}$ indicates that $\ev_V$ is included exactly when $\pi(\omega_1)= \pi(\omega_2)=-1$. The equivalence of this interpretation with that of homotopy fixed points follows from noting that a homotopy fixed point $((V,\rho_V),\psi_V)$ determines an extension of $\rho_V$ to $\Gh \setminus \G$ by
\[
\rho_V(\omega) = \hat{\theta}([\omega \varsigma^{-1} \vert \varsigma])^{-1} \rho_V(\omega \varsigma^{-1}) \circ \psi_V^{-1}, \qquad \omega \in \Gh \setminus \G.
\]

A third interpretation of twisted Real representations will also be useful.

\begin{Prop}
\label{prop:RealRepBilinearForm}
Fix $\varsigma \in \Gh \setminus \G$. A $(\hat{\theta},\lambda)$-twisted Real representation of $\G$ is equivalent to the data of a $\theta$-twisted representation of $\G$ on $V$ together with a non-degenerate bilinear form $\langle -, -\rangle : V \times V \rightarrow \C$ which satisfies the twisted $G$-invariance condition
\[
\langle \rho(g)v_1, \rho(\varsigma g \varsigma^{-1})v_2 \rangle
=
\lambda(g) \frac{\hat{\theta}([\varsigma \vert g])}{\hat{\theta}([\varsigma g \varsigma^{-1} \vert \varsigma])} \langle v_1, v_2 \rangle,
\qquad
g \in \G
\]
and the twisted symmetry condition
\[
\langle v_1, v_2 \rangle
=
\lambda(\varsigma) \theta([\varsigma^{-1} \vert \varsigma^{-1}]) \langle \rho(\varsigma^{-2})v_2, v_1 \rangle
\]
for all $v_1,v_2 \in V$.
\end{Prop}

\begin{proof}
Let $(V,\rho)$ be a $(\hat{\theta},\lambda)$-twisted Real representation of $\G$. Fix $\varsigma \in \Gh \setminus \G$ and define a non-degenerate bilinear form on $V$ by
\begin{equation}
\label{eq:bilinFormFromRealRep}
\langle v_1, v_2 \rangle
=
\rho(\varsigma^{-1})^{-1}(v_1)v_2.
\end{equation}
With this definition, $\langle \rho(g)v_1, \rho(\varsigma g \varsigma^{-1})v_2 \rangle$ is equal to
\begin{eqnarray*}
&&
\rho(\varsigma^{-1})^{-1}(\rho(g)(v_1))(\rho(\varsigma g \varsigma^{-1})v_2) \\
&=&
\lambda(\varsigma g) \hat{\theta}([\varsigma^{-1} \vert \varsigma g])^{-1} \rho(\varsigma g)^{- \vee}(\ev_V(v_1))(\rho(\varsigma g \varsigma^{-1})v_2)\\
&=&
\lambda(\varsigma g) \lambda(\varsigma^{-1}) \hat{\theta}([\varsigma^{-1} \vert \varsigma g])^{-1} \hat{\theta}([\varsigma g \vert \varsigma^{-1}])^{-1} \rho(\varsigma g \varsigma^{-1})^{- \vee} \rho(\varsigma^{-1})^{-1}(v_1)(\rho(\varsigma g \varsigma^{-1})v_2)\\
&=&
\lambda(g) \frac{\hat{\theta}([\varsigma \vert g])}{\hat{\theta}([\varsigma g \varsigma^{-1} \vert \varsigma])} \langle v_1, v_2 \rangle.
\end{eqnarray*}
The first two equalities follow from equation \eqref{eq:explicitRealRep} and the third from the $2$-cocycle condition on $\hat{\theta}$. Similarly, we compute
\begin{eqnarray*}
\langle v_1, v_2 \rangle
&=&
\lambda(\varsigma) \hat{\theta}([\varsigma^{-1} \vert \varsigma])^{-1} \rho(\varsigma)^{-\vee} (\ev_V(v_1))v_2 \\
&=&
\lambda(\varsigma)\hat{\theta}([\varsigma^{-1} \vert \varsigma])^{-1} \rho(\varsigma^{-2})^{-\vee} \circ \rho(\varsigma)^{-\vee} (\ev_V(v_1))(\rho(\varsigma^{-2})v_2) \\
&=&
\lambda(\varsigma) \hat{\theta}([\varsigma^{-1} \vert \varsigma])^{-1} \hat{\theta}([\varsigma^{-2} \vert \varsigma])^{-1} \ev_V(v_1)(\rho(\varsigma^{-1})^{-1} \circ \rho(\varsigma^{-2})v_2) \\
&=&
\lambda(\varsigma) \hat{\theta}([\varsigma^{-1} \vert \varsigma^{-1}]) \langle \rho(\varsigma^{-2})v_2,v_1 \rangle.
\end{eqnarray*}

Conversely, given $(V,\rho) \in \Rep^{\theta}(\G)$ with non-degenerate bilinear form $\langle -, - \rangle$ satisfying the conditions of the lemma, define $\rho(\varsigma^{-1})$ by equation \eqref{eq:bilinFormFromRealRep} and set
\[
\rho(\omega) = \hat{\theta}([\omega \varsigma \vert \varsigma^{-1}])^{-1} \rho(\omega \varsigma) \circ \rho(\varsigma^{-1}), \qquad \omega \in \Gh \setminus \G.
\]
The verification that $\rho$ is a $(\hat{\theta},\lambda)$-twisted Real representation of $\G$ mirrors the calculations from the previous paragraph.
\end{proof}

A $(\hat{\theta},\lambda)$-twisted Real representation is called \emph{irreducible} if it has no non-trivial Real subrepresentations. The direct sum $(V,\psi_V) \oplus (W,\psi_W) = (V \oplus W, \psi_V \oplus \psi_W)$ allows for the following formulation of a Real analogue of Maschke's lemma.

\begin{Prop}
\label{prop:RealRestr}
Let $V \in \RRep^{(\hat{\theta},\lambda)}(\G)$ be irreducible. Then the restriction of $V$ to $\G$ is irreducible or of the form $U \oplus P^{(\hat{\theta},\lambda)}(U)$ for an irreducible $U \in \Rep^{\theta}(\G)$.
\end{Prop}

\begin{proof}
Interpret $V$ as a $\theta$-twisted representation of $\G$ with compatible bilinear form $\langle-,-\rangle$, as in Proposition \ref{prop:RealRepBilinearForm}, and suppose that the restriction $V_{\vert \G}$ has a non-trivial irreducible $\theta$-twisted subrepresentation $U$. The twisted $\G$-invariance of $\langle-,-\rangle$ implies that the orthogonal complement $U^{\perp}$ is a $\theta$-twisted subrepresentation of $V_{\vert \G}$ and $V_{\vert \G} = U \oplus U^{\perp}$ as $\theta$-twisted representations. Since $V$ is irreducible, the map $\rho(\varsigma): V^{\vee} \rightarrow V$ restricts to a map $\rho(\varsigma): U^{\vee} \rightarrow U^{\perp}$ which defines an isomorphism $P^{(\hat{\theta},\lambda)}(U) \xrightarrow[]{\sim} U^{\perp}$ of $\theta$-twisted representations.
\end{proof}


\subsection{A Frobenius--Schur indicator}
\label{sec:FSIndDetail}

Keep the notation of Section \ref{sec:RealProjRep}.

\begin{Def}
The \emph{$(\hat{\theta},\lambda)$-twisted Frobenius--Schur element} is
\[
\nu_{(\hat{\theta},\lambda)}
=
\sum_{\varsigma \in \Gh \setminus \G} \frac{\lambda(\varsigma)}{\hat{\theta}([\varsigma \vert \varsigma])} l_{\varsigma^2}
\in \C^{\theta^{-1}}[\G].
\]
\end{Def}

When $(\hat{\theta},\lambda)$ is clear from the context, we write $\nu$ for $\nu_{(\hat{\theta},\lambda)}$. Note that
\[
\nu_{(\hat{\theta},\lambda)}=-\nu_{(\delta\hat{\theta},\lambda)}=-\nu_{(\hat{\theta},\pi\lambda)}.
\]

\begin{Lem}
\label{lem:FSClassFun}
The element $\nu_{(\hat{\theta},\lambda)}$ is a $\theta$-twisted class function on $\G$.
\end{Lem}

\begin{proof}
The statement amounts to the identity
\[
\hat{\theta}([h \varsigma h^{-1} \vert h \varsigma h^{-1}])^{-1}
=
\uptau(\theta)([h] \varsigma^2) \hat{\theta}([\varsigma \vert \varsigma])^{-1},
\qquad h \in \G, \; \varsigma \in \Gh \backslash \G,
\]
which is seen to hold using equation \eqref{eq:oddConj}.
\end{proof}

We require the following elementary result from linear algebra.

\begin{Lem}
\label{lem:traceTensorDual}
Let $V$ be a finite dimensional vector space and $\phi \in \Hom_{\C}(V,V)$. Then $\tr_V \,\phi$ is equal to the trace of the map
\[
\iota_{\phi}: \Hom_{\C}(V, V^{\vee}) \rightarrow \Hom_{\C}(V, V^{\vee}), \qquad f \mapsto f^{\vee} \circ \ev_V \circ \phi.
\]
\end{Lem}

\begin{proof}
Let $\dim_{\C} V= v$. Fix a basis of $V$ with induced basis $\{E_{ij}\}_{i,j=1}^v$ of $\Hom_{\C}(V,V)$. Writing $\phi = \sum_{i,j=1}^v \phi_{ij} E_{ij}$ in this basis, we compute $\iota_{\phi} (E_{ij}) = \sum_{k=1}^v \phi_{ik} E_{jk}$ so that
\[
\tr_{\Hom_{\C}(V,V^{\vee})} \, \iota_{\phi}
=
\sum_{i,j=1}^v \iota_{\phi} (E_{ij})_{ij}
=
\sum_{i,j,k=1}^v \phi_{ik}(E_{jk})_{ij}
=
\sum_{i,j,k=1}^v \phi_{ik} \delta_{ji} \delta_{kj}
=
\tr_V \, \phi. \qedhere
\]
\end{proof}

Let $V \in \Rep^{\theta}(\G)$ and $\phi \in \Hom_{\G}(V,V)$. Consider the map
\[
\iota_{\phi}: \Hom_{\G}(V,P^{(\hat{\theta},\lambda,\varsigma)}(V)) \rightarrow \Hom_{\G}(V,P^{(\hat{\theta},\lambda,\varsigma)}(V)),
\qquad
f \mapsto P^{(\hat{\theta},\lambda,\varsigma)}(f) \circ \Theta^{(\hat{\theta},\lambda,\varsigma)}_V \circ \phi.
\]
Independence of the duality structure up to equivalence on $\varsigma \in \Gh \setminus \G$ implies that $\iota_{\phi}$ is independent of $\varsigma$. The coherence condition \eqref{eq:catWDualCoher} implies that $\iota:= \iota_{\id_V}$ is an involution.

For each $V \in \Rep^{\theta}(\G)$, define
\[
\bb^V : \Hom_{\G}(V,V) \rightarrow Z(\C^{\theta^{-1}}[\G]),
\qquad
\phi \mapsto \sum_{g \in \G}  \tr_V (\phi \circ \rho_V(g)) l_g.
\]
Note that $\bb^V(\id_V) = \chi_V$.

\begin{Thm}
\label{thm:LefGroupAlg}
For each $V \in \Rep^{\theta}(\G)$ and $\phi \in \Hom_{\G}(V,V)$, there is an equality
\[
\tr_{\Hom_{\G}(V,P^{(\hat{\theta},\lambda)}(V))} \, \iota_{\phi}
=
\langle \bb^V(\phi), \nu_{(\hat{\theta},\lambda)} \rangle_{\G}.
\]
\end{Thm}

\begin{proof}
Write $(P, \Theta)$ for $(P^{(\hat{\theta},\lambda,\varsigma)},\Theta^{(\hat{\theta},\lambda,\varsigma)})$. We compute
\begin{eqnarray*}
\tr_{\Hom_{\G}(V,P(V))} \, \iota_{\phi}
&=&
\tr_{\Hom_{\G}(V,P(V))} \, (f \mapsto \frac{\lambda(\varsigma)}{\hat{\theta}([\varsigma \vert \varsigma])} f^{\vee} \circ \ev_V \circ \rho(\varsigma^2)^{-1} \circ \phi) \\
&=&
\frac{\lambda(\varsigma)}{\hat{\theta}([\varsigma \vert \varsigma])} \tr_V\,(\rho(\varsigma^2)^{-1} \circ \phi),
\end{eqnarray*}
the second equality following from Lemma \ref{lem:traceTensorDual}. Since $\tr_{\Hom_{\G}(V,P(V))} \, \iota_{\phi}$ is independent of the choice $\varsigma \in \Gh \setminus \G$ used in the definition of $\iota_{\phi}$, we average over all such choices to obtain
\[
\tr_{\Hom_{\G}(V,P(V))} \, \iota_{\phi}
=
\frac{1}{\vert \G \vert} \sum_{\varsigma \in \Gh \backslash \G}\frac{\lambda(\varsigma)}{\hat{\theta}([\varsigma \vert \varsigma])} \tr_V\,(\rho(\varsigma^2)^{-1} \circ \phi).
\]
On the other hand, we have
\begin{eqnarray*}
\langle \bb^V(\phi), \nu \rangle_{\G}
&=&
\frac{1}{\vert \G \vert} \sum_{\varsigma \in \Gh \backslash \G} \frac{\lambda(\varsigma)}{\hat{\theta}([\varsigma \vert \varsigma])} \theta([\varsigma^2 \vert \varsigma^{-2}])^{-1} \tr_V\, \left(\phi \circ \rho_V(\varsigma^{-2}) \right) \\
&=&
\frac{1}{\vert \G \vert} \sum_{\varsigma \in \Gh \backslash \G} \frac{\lambda(\varsigma)}{\hat{\theta}([\varsigma \vert \varsigma])} \tr_V\, \left(\rho_V(\varsigma^{2})^{-1} \circ \phi \right),
\end{eqnarray*}
thereby proving the desired equality.
\end{proof}

Recall that $\delta$ is a cocycle representative of the generator of $H^{2+\pi}(B \Ctwo) \simeq \Ctwo$.

\begin{Cor}
\label{cor:FSInd}
Let $V$ be an irreducible $\theta$-twisted representation of $\G$. Then
\[
\langle \chi_V, \nu_{(\hat{\theta},\lambda)} \rangle_{\G}
=
\begin{cases}
1 & \mbox{if and only if $V$ lifts to a $(\hat{\theta},\lambda)$-twisted Real representation},\\
-1 & \mbox{if and only if $V$ lifts to a $(\delta\hat{\theta},\lambda)$-twisted Real representation}, \\
0 & \mbox{otherwise}.
\end{cases}
\]
When $\langle \chi_V, \nu_{(\hat{\theta},\lambda)} \rangle_{\G} = \pm 1$, the twisted Real structure on $V$ is unique up to isomorphism.
\end{Cor}

\begin{proof}
Schur's Lemma for $\Rep^{\theta}(\G)$ implies that $\Hom_G(V, P^{(\hat{\theta},\lambda)}(V)) \simeq \C$ if $P^{(\hat{\theta},\lambda)}(V) \simeq V$ and $\Hom_G(V,P^{(\hat{\theta},\lambda)}(V))=0$ otherwise. Hence, if $\Hom_G(V,P^{(\hat{\theta},\lambda)}(V))=0$, then $V$ does not lift to a Real representation and the statement follows by applying Theorem \ref{thm:LefGroupAlg} with $\phi = \id_V$. If $\Hom_G(V, P^{(\hat{\theta},\lambda)}(V)) \simeq \C$, then a non-zero element $\psi_V \in \Hom_G(V, P^{(\hat{\theta},\lambda)}(V))$ is an isomorphism which, by Theorem \ref{thm:LefGroupAlg}, satisfies
\[
P^{(\hat{\theta},\lambda)}(\psi_V) \circ \Theta_V = \langle \chi_V, \nu_{(\hat{\theta},\lambda)} \rangle_{\G} \cdot \psi_V.
\]
The first statement of the corollary now follows from the homotopy fixed point interpretation of twisted Real representations. Uniqueness of the Real structure up to isomorphism follows from one dimensionality of $\Hom_G(V, P^{(\hat{\theta},\lambda)}(V))$. 
\end{proof}

In the setting of Corollary \ref{cor:FSInd}, if $\langle \chi_V, \nu_{(\hat{\theta},\lambda)} \rangle_{\G} =0$, then $V$ determines an irreducible $(\hat{\theta},\lambda)$-twisted Real representation by $H^{(\hat{\theta},\lambda)}(V) = V \oplus P^{(\hat{\theta},\lambda)}(V)$ with its hyperbolic homotopy fixed point structure \cite[\S 7.3]{mbyoung2021b}. Note that $H^{(\hat{\theta},\lambda)}(V) \simeq H^{(\hat{\theta},\lambda)}(P^{(\hat{\theta},\lambda)}(V))$.

\begin{Cor}
\label{cor:finiteRealIrreps}
There are finitely many isomorphism classes of irreducible $(\hat{\theta},\lambda)$-twisted Real representations.
\end{Cor}

\begin{proof}
This follows from Proposition \ref{prop:RealRestr}, finiteness of isomorphism classes of irreducible $\theta$-twisted representations and the final statement of Corollary \ref{cor:FSInd}.
\end{proof}

\begin{Cor}
\label{cor:FSIndDecomp}
There is an equality
\[
\nu_{(\hat{\theta},\lambda)}
=
\sum_{\substack{ V \in \Irr^{\theta}(\G) \\ P^{(\hat{\theta},\lambda)}(V) \simeq V}} \langle \chi_V, \nu_{(\hat{\theta},\lambda)} \rangle_G \chi_V.
\]
\end{Cor}

\begin{proof}
This follows from the second part of Corollary \ref{cor:FSInd} and the fact that the set $\{\chi_V\}_{V \in \Irr^{\theta}(\G)}$ of irreducible $\theta$-twisted characters is an orthonormal basis of the space of $\theta$-twisted class function on $\G$.
\end{proof}

Various instances of the element $\nu_{(\hat{\theta},\lambda)}$ and Corollary \ref{cor:FSInd} are known:
\begin{enumerate}
\item The classical setting of Frobenius and Schur \cite{frobenius1906} corresponds to taking $\Gh = \G \times \Ctwo$ with $\pi$ the projection to the second factor and the cohomological data $\hat{\theta}$ and $\lambda$ trivial. The conditions on the bilinear form $\langle -, - \rangle$ from Proposition \ref{prop:RealRepBilinearForm} reduce to $\G$-invariance and symmetry; if $\hat{\theta}=\delta$, then $\langle -, - \rangle$ is skew-symmetric. Corollary \ref{cor:FSInd} then gives the standard necessary and sufficient condition for $V$ to admit a $\G$-invariant bilinear form and so be defined over $\R$ (in the symmetric case) or $\quat$ (in the skew-symmetric case).

\item Taking $\hat{\theta}$ and $\lambda$ to be trivial recovers Gow's generalized Frobenius--Schur element used in the character theoretic study of $2$-regularity of finite groups \cite[\S 2]{gow1979}. For representation theoretic applications, see \cite{rumynin2021}.

\item Take $\Gh = \G \times \Ctwo$. In this case, there is a homomorphism $\G \rightarrow \Gh$ which splits $\pi$. When $\hat{\theta}$ is in the image of the resulting map $H^2(B \G ; \Ctwo) \rightarrow H^{2+\pi}(B \Gh)$ and $\lambda$ is trivial, $\nu_{(\hat{\theta},1)}$ recovers Turaev's generalized Frobenius--Schur element studied in the context of closed unoriented TFT \cite{turaev2007}. See \cite{ichikawa2023} for a generalization in the setting of closed $\Pin_2^-$ TFT.

\item When $\phi=\id_V$ the trace of Theorem \ref{thm:LefGroupAlg} is an instance of a Shimizu's Frobenius--Schur indicator in a category with duality \cite{shimizu2012}.
\end{enumerate}

\begin{Ex}
\label{ex:cyclicGroups}
Let $\G = C_n$ with generator $r$ and $\zeta= e^{\frac{2\pi \sqrt{-1}}{n}}$. The one dimensional representations $\{\rho_k \mid 0 \leq k \leq n-1\}$, defined by $\rho_k(r)=\zeta^k$, constitute a complete set of irreducible representations of $\G$. Take $\hat{\theta}$ and $\lambda$ to be trivial in this example.
\begin{enumerate}
\item \label{ite:trivRealStr} Let $\Gh = C_n \times \Ctwo$ with $\pi$ projection to the second factor. We have 
\[
\langle \chi_k, \nu \rangle_{\G} 
= \begin{cases}
1 & \mbox{if $k=0$ or $k=\frac{n}{2}$}, \\
0 & \text{otherwise},
\end{cases}
\]
whence the trivial and sign representation (which exists when $n$ is even) admit Real structures. These are precisely the irreducible representations which are defined over $\R$.

\item \label{ite:cyclicRealStr} Let $\Gh = C_{2n}$ with generator $\varsigma$ satisfying $\varsigma^2=r$ and $\Ctwo$-grading $\pi:\Gh \rightarrow \Ctwo$ determined by $\pi(\varsigma)=-1$. Assume that $n$ is even, as otherwise $\Gh \simeq C_n \times \Ctwo$ as $\Ctwo$-graded groups. We have $\nu = 2 \sum_{j=0}^{\frac{n}{2}-1} l_{r^{2j}}$ from which we compute
\[
\langle \chi_k, \nu \rangle_{\G}
=
\frac{2}{n} \sum_{j=0}^{\frac{n}{2}-1} \zeta^{2kj} \\
=
\begin{cases}
1 & \mbox{if } k=0, \\
-1 &  \mbox{if } k=\frac{n}{2}, \\
0 & \mbox{otherwise}.
\end{cases}
\]
The Real structure on $\rho_0$ is given by $\rho_{0}(\varsigma)(1^{\vee}) = 1$. The same formula gives the $\delta$-twisted Real structure on $\rho_{\frac{n}{2}}$.

\item  \label{ite:dihedralRealStr} Let $\Gh$ be the dihedral group $D_{2n} = \langle r,s \mid r^n=s^2=e, \, srs=r^{-1} \rangle$ with $\pi: \Gh \rightarrow \Ctwo$ determined by $\pi(r)=1$ and $\pi(s)=-1$.
We have $\nu = nl_{e}$ from which we compute $\langle \chi_k, \nu \rangle_{\G} =1$. Each irreducible representation $\rho_k$ can therefore be extended to a Real representation by the formula $\rho_k(s)(1^{\vee})=1$.\qedhere
\end{enumerate}
\end{Ex}

\begin{Ex}
Let $\Gh=Q_8$ be the quaternion group with $\Ctwo$-grading given on the standard generators by $\pi(i)=1$ and $\pi(j)=-1$. Then $G \simeq C_4$ is generated by $i$. We have $\nu = 4 l_{-1}$ so that $\langle \chi_k, \nu \rangle = (-1)^k$. The Real structure on $\rho_k$, which is $\delta$-twisted precisely when $k$ is even, is determined by $\rho_k(j)(1^{\vee}) = 1$.
\end{Ex}

\begin{Ex}
Let $\G =A_4$ be the alternating group on $4$ letters. The irreducible representations of $\G$ are the trivial representation $U$, two non-trivial one dimensional representations $U^{\prime}$ and $U^{\prime \prime}$ and a three dimensional representation $V$. Writing $\zeta= e^{\frac{2\pi \sqrt{-1}}{3}}$, we take the convention that their characters are
\begin{align*}
\chi_{U^{\prime}}(123) & = \zeta, & \chi_{U^{\prime}}(132) &= \zeta^2, &
\chi_{U^{\prime}}((12)(34))& =1, \\
\chi_{U^{\prime \prime}}(123) & = \zeta^2, & \chi_{U^{\prime \prime}}(132) &= \zeta, & \chi_{U^{\prime\prime}}((12)(34)) & =1, \\
\chi_V(123) & =0, & \chi_V(132) & =0, & \chi_V((12)(34)) & = -1.
\end{align*}
\begin{enumerate}
\item Taking $\Gh = A_4 \times \Ctwo$ with $\pi$ the projection to the second factor gives $\nu= 4 l_{(1)} + \sum_{3\mbox{\tiny-cycles} \; \sigma} l_{\sigma}$. Using this, we compute
\[
\langle \chi_U,\nu \rangle = 1,
\qquad
\langle \chi_{U^{\prime}},\nu \rangle=0,
\qquad
\langle \chi_{U^{\prime \prime}},\nu \rangle=0,
\qquad
\langle \chi_V,\nu \rangle= 1.
\]
Hence, only $U$ and $V$ admit real structures.

\item Taking $\Gh = S_4$ the symmetric group with $\pi$ the sign representation gives $\nu = 6 l_{(1)} +2 (l_{(12)(34)} + l_{(13)(24)} + l_{(14)(23)})$. Using this, we compute
\[
\langle \chi_U,\nu \rangle = 1,
\qquad
\langle \chi_{U^{\prime}},\nu \rangle=1,
\qquad
\langle \chi_{U^{\prime \prime}},\nu \rangle=1,
\qquad
\langle \chi_V,\nu \rangle= 0.
\]
Hence, all one dimensional representations admit Real structures. Taking $\lambda$ to be non-trivial, that is, $\lambda=\pi$, replaces $\nu$ with its negative and leads to $\delta$-twisted Real structures on the one dimensional representations.\qedhere
\end{enumerate}
\end{Ex}

\section{Two dimensional unoriented open/closed topological field theory}
\label{sec:TFT}

\subsection{Algebraic characterization}

Following Lazaroiu \cite{lazaroiu2001} and Moore and Segal \cite{moore2006}, we begin by recalling an algebraic characterization of two dimensional oriented open/closed topological field theories (TFTs). See also \cite{alexeevski2006,lauda2008}. In topological terms, such a TFT is a symmetric monoidal functor $\TFT: \Bord_2^{\ori,D} \rightarrow \Vect_{\C}$. Here $\Bord_2^{\ori,D}$ two dimensional open/closed bordism category \cite[\S 3]{lauda2008}. Objects are compact oriented $1$-manifolds with boundary components labelled by elements of a given set $D$. Morphisms are isomorphism classes of oriented bordisms with corners whose free boundaries are $D$-labelled compatibly with the incoming and outgoing boundaries. The monoidal structure of $\Bord_2^{\ori,D}$ is disjoint union.

\begin{Thm}[{\cite[Theorem 1]{moore2006}}]
\label{thm:ocOriTFT}
Two dimensional oriented open/closed TFTs are classified by the following data:
\begin{enumerate}
\item A commutative Frobenius algebra $A$ with identity $1_A$ and trace $\langle- \rangle_{0}: A \rightarrow \C$.
\item A  Calabi--Yau category $\Brane$, that is, $\C$-linear additive category with cyclic traces $\langle-\rangle_V: \Hom_{\Brane}(V,V) \rightarrow \C$, $V \in \Brane$, whose associated pairings 
\[
\langle -, - \rangle_{V,W}:
\Hom_{\Brane}(W,V) \otimes \Hom_{\Brane}(V,W) \xrightarrow[]{\circ} \Hom_{\Brane}(V,V) \xrightarrow[]{\langle - \rangle_V} \C
\]
are non-degenerate.

\item For each $V \in \Brane$, a linear \emph{boundary-bulk} map $\bb^V : \Hom_{\Brane}(V,V) \rightarrow A$ 
and linear \emph{bulk-boundary} map $\bb_V: A \rightarrow \Hom_{\Brane}(V,V)$.
\end{enumerate}
This data is required to satisfy the following conditions:
\begin{enumerate}[label=(\roman*)]
\item $\bb_V$ is a unital algebra homomorphism.


\item $\bb_W(a) \circ \phi = \phi \circ \bb_V(a)$ for all $a \in A$ and $\phi \in \Hom_{\Brane}(V,W)$.

\item $\langle \phi, \bb_V (a) \rangle_{V,V} = \langle \bb^V(\phi), a \rangle_0$ for all $a \in A$ and $\phi \in \Hom_{\Brane}(V,V)$.

\item (The \emph{oriented Cardy condition}) Let $\{\psi_i\}_i$ be a basis of $\Hom_{\Brane}(V,W)$ and $\{\psi^i\}_i$ the basis of $\Hom_{\Brane}(W,V)$ which is dual with respect to $\langle-,-\rangle_{V,W}$. Then $\bb_V \circ \bb^W$ is equal to the map
\[
\Hom_{\Brane}(W,W) \rightarrow \Hom_{\Brane}(V,V), \qquad \phi \mapsto \sum_i \psi^i \circ \phi \circ \psi_i.
\]
\end{enumerate}
\end{Thm}

\begin{Rems}
\begin{enumerate}

\item When $\Brane$ has a single object, the algebraic data of Theorem \ref{thm:ocOriTFT} is called a \emph{Cardy--Frobenius} or \emph{knowledgeable} Frobenius algebra \cite{alexeevski2006, lauda2008}.

\item Let $\TFT$ be an oriented open/closed TFT with object set $D$. The category\footnote{Since $\Brane$ is assumed to be additive, it may be required to formally add some elements to $D$ to ensure the existence of direct sums. See \cite[\S 2.5]{moore2006}.} $\Brane$ has objects $D$, morphisms $\Hom_{\Brane}(V,W)$ given by the value of $\TFT$ on the closed interval labelled by $V$ and $W$ and oriented from $V$ to $W$ and composition defined by the value of $\TFT$ on the flattened pair of pants. The value of $\TFT$ on the flattened cap defines the Calabi--Yau traces.

\item By non-degeneracy of the Calabi--Yau pairings, the oriented Cardy condition holds if and only if
\begin{equation}
\label{eq:baggyCardy}
\tr_{\Hom_{\Brane}(V,W)} \, (f \mapsto \psi \circ f \circ \phi) = \langle \bb^W(\psi), \bb^V(\phi) \rangle_0
\end{equation}
for all $\phi \in \Hom_{\Brane}(V,V)$ and $\psi \in \Hom_{\Brane}(W,W)$. Following \cite[\S 7.4]{caldararu2010}, we refer to equation \eqref{eq:baggyCardy} as the \emph{baggy oriented Cardy condition}. Topologically, the oriented Cardy condition asserts the equality of two ways of evaluating the TFT on the annulus with boundary components labelled by $V$ and $W$.
\end{enumerate}
\end{Rems}

We are interested in the extension of Theorem \ref{thm:ocOriTFT} to the unoriented bordism category $\Bord_2^D$, defined analogously to $\Bord_2^{\ori,D}$ except that objects and morphisms are unoriented. Upon restriction to the closed sector, the extension is known.

\begin{Thm}[{\cite[Proposition 2.9]{turaev2006}}]
\label{thm:cUnoriTFT}
Two dimensional unoriented TFTs are classified by the data of an \emph{unoriented Frobenius algebra}, that is, a commutative Frobenius algebra $(A,1_A, \langle - \rangle_0)$ with an isometric algebra involution $p: A \rightarrow A$ and an element $Q \in A$, the \emph{crosscap state}, which satisfy the following conditions:
\begin{enumerate}[label=(\roman*)]
\item $p(Q a) = Q a$ for all $a \in A$.

\item (\emph{The Klein condition}) Given a basis $\{a_i\}_i$ of $A$ with basis $\{a^i\}_i$ of $A$ dual with respect to $\langle - \rangle_0$, the equality $Q^2 = \sum_i p(a^i ) a_i$ holds.
\end{enumerate}
\end{Thm}

In terms of bordisms, $Q$ is the image under $\TFT$ of the compact M\"{o}bius strip $\mathbb{RP}^2 \setminus \mathring{D}^2$,
\[
\begin{tikzpicture}[very thick,scale=2.5,color=black,baseline=0.75cm]
\coordinate (q1) at (-0.4,0.175);
\coordinate (q2) at (-0.4,0.525);
\coordinate (x1) at (-0.63,0.375);
\coordinate (x2) at (-0.58,0.325);
\coordinate (y1) at (-0.63,0.325);
\coordinate (y2) at (-0.58,0.375);
\draw[very thick] (q1) .. controls +(-0.4,0) and +(-0.4,0) ..  (q2); 
\draw[very thick, blue!80!black,decoration={markings, mark=at position 0.5 with {\arrow{>}}}, postaction={decorate}] (q1) .. controls +(0.15,0) and +(0.15,0) ..  (q2); 
\draw[very thick, blue!80!black, opacity=0.2] (q1) .. controls +(-0.15,0) and +(-0.15,0) ..  (q2); 
\draw[thick] (x1) to (x2);
\draw[thick] (y1) to (y2);
\draw[thick] (-0.605,0.35) circle (0.05cm);
\end{tikzpicture}
: \varnothing \rightarrow S^1,
\]
and $p$ is the image of the mapping cylinder of circle reflection,
\[
\begin{tikzpicture}[very thick,scale=2.5,color=black,baseline]
\coordinate (r1) at (0.1,0.15);
\coordinate (r2) at (0.1,-0.15);
\coordinate (r3) at (0.8,0.15);
\coordinate (r4) at (0.8,-0.15);
\draw (r1) to (r3);
\draw (r2) to (r4);
\draw[very thick, blue!80!black,decoration={markings, mark=at position 0.5 with {\arrow{<}}}, postaction={decorate}] (r1) .. controls +(0.15,0) and +(0.15,0) ..  (r2); 
\draw[very thick, blue!80!black] (r1) .. controls +(-0.15,0) and +(-0.15,0) ..  (r2); 
\draw[very thick, blue!80!black,decoration={markings, mark=at position 0.5 with {\arrow{>}}}, postaction={decorate}] (r3) .. controls +(0.15,0) and +(0.15,0) ..  (r4); 
\draw[very thick, blue!80!black,opacity=0.2] (r3) .. controls +(-0.15,0) and +(-0.15,0) ..  (r4);
\end{tikzpicture}
: S^1 \rightarrow S^1.
\]
The Klein condition is illustrated in Figure \ref{fig:KleinCond}.

\begin{figure}
\begin{tikzpicture}[very thick,scale=2.01,baseline=-0.10cm,color=black]
\coordinate (q1) at (-0.4,0.175);
\coordinate (q2) at (-0.4,0.525);
\coordinate (x1) at (-0.63,0.375);
\coordinate (x2) at (-0.58,0.325);
\coordinate (y1) at (-0.63,0.325);
\coordinate (y2) at (-0.58,0.375);
\draw[very thick] (q1) .. controls +(-0.4,0) and +(-0.4,0) ..  (q2); 
\draw[very thick, blue!80!black,decoration={markings, mark=at position 0.5 with {\arrow{>}}}, postaction={decorate}] (q1) .. controls +(0.15,0) and +(0.15,0) ..  (q2); 
\draw[very thick, blue!80!black, opacity=0.2] (q1) .. controls +(-0.15,0) and +(-0.15,0) ..  (q2); 
\draw[thick] (x1) to (x2);
\draw[thick] (y1) to (y2);
\draw[thick] (-0.605,0.35) circle (0.05cm);
\coordinate (p1) at (0,-0.575);
\coordinate (p2) at (0,-0.225);
\coordinate (p3) at (0,0.175);
\coordinate (p4) at (0,0.525);
\coordinate (p5) at (0.8,0.15);
\coordinate (p6) at (0.8,-0.15);
\draw (p2) .. controls +(0.35,0) and +(0.35,0) ..  (p3); 
\draw (p4) .. controls +(0.5,0) and +(-0.5,0) ..  (p5); 
\draw (p6) .. controls +(-0.5,0) and +(0.5,0) ..  (p1); 
\draw[very thick, blue!80!black,decoration={markings, mark=at position 0.5 with {\arrow{>}}}, postaction={decorate}] (p1) .. controls +(0.15,0) and +(0.15,0) ..  (p2); 
\draw[very thick, blue!80!black] (p1) .. controls +(-0.15,0) and +(-0.15,0) ..  (p2); 
\draw[very thick, blue!80!black,decoration={markings, mark=at position 0.5 with {\arrow{>}}}, postaction={decorate}] (p3) .. controls +(0.15,0) and +(0.15,0) ..  (p4); 
\draw[very thick, blue!80!black] (p3) .. controls +(-0.15,0) and +(-0.15,0) ..  (p4); 
\draw[very thick, blue!80!black,decoration={markings, mark=at position 0.5 with {\arrow{<}}}, postaction={decorate}] (p5) .. controls +(0.15,0) and +(0.15,0) ..  (p6); 
\draw[very thick, blue!80!black,opacity=0.2] (p5) .. controls +(-0.15,0) and +(-0.15,0) ..  (p6);
\end{tikzpicture}
\;\;
=
\;\;
\begin{tikzpicture}[very thick,scale=1.75,color=black,baseline]
\coordinate (q1) at (-0.4,0.175);
\coordinate (q2) at (-0.4,0.525);
\coordinate (x1) at (-0.63,0.375);
\coordinate (x2) at (-0.58,0.325);
\coordinate (y1) at (-0.63,0.325);
\coordinate (y2) at (-0.58,0.375);
\draw[very thick] (q1) .. controls +(-0.4,0) and +(-0.4,0) ..  (q2); 
\draw[very thick, blue!80!black,decoration={markings, mark=at position 0.5 with {\arrow{>}}}, postaction={decorate}] (q1) .. controls +(0.15,0) and +(0.15,0) ..  (q2); 
\draw[very thick, blue!80!black, opacity=0.2] (q1) .. controls +(-0.15,0) and +(-0.15,0) ..  (q2); 
\draw[thick] (x1) to (x2);
\draw[thick] (y1) to (y2);
\draw[thick] (-0.605,0.35) circle (0.05cm);
\coordinate (p1) at (0,-0.575);
\coordinate (p2) at (0,-0.225);
\coordinate (p3) at (0,0.175);
\coordinate (p4) at (0,0.525);
\coordinate (p5) at (0.8,0.15);
\coordinate (p6) at (0.8,-0.15);
\draw (p2) .. controls +(0.35,0) and +(0.35,0) ..  (p3); 
\draw (p4) .. controls +(0.5,0) and +(-0.5,0) ..  (p5); 
\draw (p6) .. controls +(-0.5,0) and +(0.5,0) ..  (p1); 
\draw[very thick, blue!80!black,decoration={markings, mark=at position 0.5 with {\arrow{>}}}, postaction={decorate}] (p1) .. controls +(0.15,0) and +(0.15,0) ..  (p2); 
\draw[very thick, blue!80!black] (p1) .. controls +(-0.15,0) and +(-0.15,0) ..  (p2); 
\draw[very thick, blue!80!black,decoration={markings, mark=at position 0.5 with {\arrow{>}}}, postaction={decorate}] (p3) .. controls +(0.15,0) and +(0.15,0) ..  (p4); 
\draw[very thick, blue!80!black] (p3) .. controls +(-0.15,0) and +(-0.15,0) ..  (p4); 
\draw[very thick, blue!80!black,decoration={markings, mark=at position 0.5 with {\arrow{<}}}, postaction={decorate}] (p5) .. controls +(0.15,0) and +(0.15,0) ..  (p6); 
\draw[very thick, blue!80!black,opacity=0.2] (p5) .. controls +(-0.15,0) and +(-0.15,0) ..  (p6); 
\coordinate (r1) at (1.2,0.15);
\coordinate (r2) at (1.2,-0.15);
\coordinate (r3) at (1.9,0.15);
\coordinate (r4) at (1.9,-0.15);
\draw (r1) to (r3);
\draw (r2) to (r4);
\draw[very thick, blue!80!black,decoration={markings, mark=at position 0.5 with {\arrow{<}}}, postaction={decorate}] (r1) .. controls +(0.15,0) and +(0.15,0) ..  (r2); 
\draw[very thick, blue!80!black] (r1) .. controls +(-0.15,0) and +(-0.15,0) ..  (r2); 
\draw[very thick, blue!80!black,decoration={markings, mark=at position 0.5 with {\arrow{>}}}, postaction={decorate}] (r3) .. controls +(0.15,0) and +(0.15,0) ..  (r4); 
\draw[very thick, blue!80!black,opacity=0.2] (r3) .. controls +(-0.15,0) and +(-0.15,0) ..  (r4); 
\end{tikzpicture}
\caption{The equality of bordisms responsible for the Klein condition.}
\label{fig:KleinCond}
\end{figure}
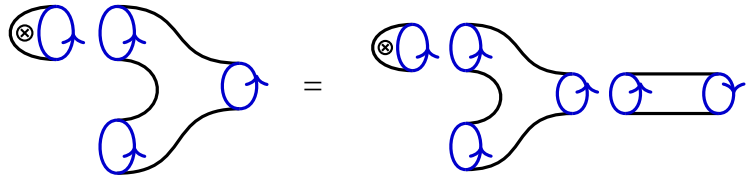

We now come the main classification result.

\begin{Thm}
\label{thm:genStruAlg}
Two dimensional unoriented open/closed TFTs are classified by the data of an underlying closed theory, as in Theorem \ref{thm:cUnoriTFT}, together with the data of a $\C$-linear strict duality $P$ on $\mathcal{B}$. This data is required to satisfy the following conditions:
\begin{enumerate}[label=(\roman*)]
\item \label{ite:identObj} The functor $P$ is the identity on objects.

\item \label{ite:compatPairing} $\langle P(\phi) \rangle_V = \langle \phi \rangle_V$ for all $\phi \in \Hom_{\Brane}(V,V)$.

\item \label{ite:compatBulkBound} $P \circ \bb_V = \bb_V \circ p$ for all $V \in \Brane$.

\item \label{ite:compatBoundBulk} $p \circ \bb^V = \bb^V\circ P$ for all $V \in \Brane$.

\item \label{ite:unoriCardy} (The \emph{unoriented Cardy condition}) Let $\{\psi_i\}_i$ be a basis of $\Hom_{\Brane}(V,V)$ with dual basis $\{\psi^i\}_{i}$ with respect to $\langle-,-\rangle_{V,V}$. Then there is an equality
\begin{equation}
\label{eq:unoriCardy}
\bb_V(Q) = \sum_i \psi^i \circ P(\psi_i).
\end{equation}
\end{enumerate}
\end{Thm}

\begin{proof}
The theorem is proved in \cite[\S 4]{alexeevski2006} under the assumption that $\Brane$ has a single object, where the above algebraic data is known as a \emph{structure algebra}. This proof generalizes immediately to allow for $\Brane$ to have many objects, in the same way as the analogous generalization in the oriented case \cite[\S 5]{lauda2008}.
\end{proof}

Topologically, $P$ is the image under $\TFT$ of the mapping cylinder of reflection of the closed interval so that $P_{V,W}: \Hom_{\Brane}(V,W) \rightarrow \Hom_{\Brane}(W,V)$ comes from the bordism
\[
\begin{tikzpicture}[scale=3.0]
\draw[very thick, blue!80!black,dashed] (-.25,-1) to (0.25,-1.0);
\draw[very thick, blue!80!black,dashed] (-.25,-0.75) to (0.25,-0.75);
\draw[very thick, blue!80!black,decoration={markings, mark=at position 0.5 with {\arrow{>}}}, postaction={decorate}] (-.25,-1) to (-.25,-0.75);
\draw[very thick, blue!80!black,decoration={markings, mark=at position 0.5 with {\arrow{<}}}, postaction={decorate}] (.25,-1) to (.25,-0.75);
\node at (-.25,-1.05)  {\scriptsize $V$};
\node at (-.25,-0.7)  {\scriptsize $W$};
\node at (.25,-1.05)  {\scriptsize $V$};
\node at (.25,-0.7)  {\scriptsize $W$};
\node at (0.375,-0.9)  {=};
\end{tikzpicture}
\begin{tikzpicture}[scale=3.0]
\draw[very thick, blue!80!black,dashed] (-.25,-1) to (0.25,-0.75);
\draw[very thick, blue!80!black,dashed] (-.25,-0.75) to (0.25,-1.0);
\draw[very thick, blue!80!black,decoration={markings, mark=at position 0.5 with {\arrow{>}}}, postaction={decorate}] (-.25,-1) to (-.25,-0.75);
\draw[very thick, blue!80!black,decoration={markings, mark=at position 0.5 with {\arrow{>}}}, postaction={decorate}] (.25,-1) to (.25,-0.75);
\node at (-.25,-1.05)  {\scriptsize $V$};
\node at (-.25,-0.7)  {\scriptsize $W$};
\node at (.25,-1.05)  {\scriptsize $W$};
\node at (.25,-0.7)  {\scriptsize $V$};
\node at (0.35,-0.9)  {.};
\end{tikzpicture}
\]
As indicated on the right, we will picture this bordism as embedded in $\R^3$ with a half-twist. That $P$ is a strict involution follows from the fact that reflection of the closed interval is an involution. We record two basic consequences of Theorem \ref{thm:genStruAlg}.

\begin{Prop}
\label{prop:wittenIndices}
\hspace{2em}
\begin{enumerate}
\item The equality $\langle Q^2 \rangle_0 = \tr_A \, p$ holds.
\item For any $V \in \Brane$ and $\phi \in \Hom_{\Brane}(V,V)$, the equality
\[
\langle \bb^V(\phi), Q \rangle_0
=
\tr_{\Hom_{\Brane}(V,V)} \, \iota_{\phi}
\]
holds, where $\iota_{\phi}$ is defined analogously to Section \ref{sec:FSIndDetail}.
\end{enumerate}
\end{Prop}

\begin{proof}
Since $p$ is an involution, there exists a basis $\{a_i\}_i$ of $A$ such that $p(a_i) = s_i a_i$ with $s_i \in \{1,-1\}$. Let $\{a^i\}_i$ be a dual basis, so that $\langle a^j, a_i \rangle_0 = \delta^j_i$. Since $p$ is an isometry of $\langle- \rangle_0$, we have $p(a^i) = s_i a^i$. With these preliminaries, we compute
\[
\langle Q^2 \rangle_0
=
\langle \sum_i p(a^i) a_i \rangle_0 = \sum_i s_i \langle a^i a_i \rangle_0 = \sum_i s_i = \tr_A \, p.
\]

For the second statement, we compute
\begin{multline*}
\langle \bb^V(\phi), Q \rangle_0
=
\langle \phi, \bb_V(Q) \rangle_{V,V}
=
\sum_i \langle \phi \circ  \psi^i \circ P(\psi_i) \rangle_V =\\
\sum_i \langle \psi^i \circ \iota_{\phi}(\psi_i)\rangle_V
=
\tr_{\Hom_{\Brane}(V,V)} \, \iota_{\phi}.
\end{multline*}
The first equality is the adjointness of $\bb^V$ and $\bb_V$, the second is the unoriented Cardy condition and the third is cyclicity of traces. 
\end{proof}

By non-degeneracy of the Calabi--Yau pairings, the unoriented Cardy condition is equivalent to the second equality from Proposition \ref{prop:wittenIndices}, which we term the \emph{baggy unoriented Cardy condition}. The unoriented Cardy condition reflects the equality of two ways of evaluating the TFT on the M\"{o}bius strip with boundary component labelled by $V$. See Figure \ref{fig:unoriCardy}.

The next result constructs the algebraic input of Theorem \ref{thm:genStruAlg} from a Calabi--Yau category with a contravariant involution which need not act trivially on objects.

\begin{figure}
\begin{tikzpicture}[very thick,scale=1.75,baseline=0.55cm,color=black]
\coordinate (q1) at (-0.4,0.175);
\coordinate (q2) at (-0.4,0.525);
\coordinate (x1) at (-0.63,0.375);
\coordinate (x2) at (-0.58,0.325);
\coordinate (y1) at (-0.63,0.325);
\coordinate (y2) at (-0.58,0.375);
\draw[very thick] (q1) .. controls +(-0.4,0) and +(-0.4,0) ..  (q2); 
\draw[very thick, blue!80!black,decoration={markings, mark=at position 0.5 with {\arrow{>}}}, postaction={decorate}] (q1) .. controls +(0.15,0) and +(0.15,0) ..  (q2); 
\draw[very thick, blue!80!black, opacity=0.2] (q1) .. controls +(-0.15,0) and +(-0.15,0) ..  (q2); 
%
\draw[thick] (x1) to (x2);
\draw[thick] (y1) to (y2);
\draw[thick] (-0.605,0.35) circle (0.05cm);
\coordinate (p1) at (0,-0.575);
\coordinate (p2) at (0,-0.225);
\coordinate (p3) at (0,0.175);
\coordinate (p4) at (0,0.525);
\coordinate (p5) at (0.8,0.15);
\coordinate (p6) at (0.8,-0.15);
%
\coordinate (r1) at (0,0.525);
\coordinate (r2) at (0.0,0.155);
\coordinate (r3) at (0.75,0.525);
\coordinate (r4) at (0.75,0.155);
%
\draw (r1) to (r3);
\draw (r2) to (r4);
\draw[very thick, blue!80!black,decoration={markings, mark=at position 0.5 with {\arrow{<}}}, postaction={decorate}] (r1) .. controls +(0.15,0) and +(0.15,0) ..  (r2); 
\draw[very thick, blue!80!black] (r1) .. controls +(-0.15,0) and +(-0.15,0) ..  (r2);
\draw[very thick, blue!80!black] (r3) to (r4);
\draw[very thick, blue!80!black,dashed] (r3) .. controls +(-0.15,0) and +(-0.15,0) ..  (r4); 
\end{tikzpicture}
\;\;
=
\;\;
\begin{tikzpicture}[very thick,scale=2.01,color=black,baseline=-1.1cm]
\draw[very thick, blue!80!black,dashed] (-.35,0) arc [start angle=90, end angle=270, x radius=0.5,y radius =0.5] ;
\draw[very thick, blue!80!black,dashed] (-.35,-0.25) arc [start angle=90, end angle=270, x radius=0.25,y radius =0.25] ;
\draw[very thick, blue!80!black] (-.35,0) to (-.35,-0.25);
\draw[very thick, blue!80!black] (-.35,-0.75) to (-.35,-1.0);
\draw[very thick, blue!80!black,dashed] (-.25,0) to (0.25,0);
\draw[very thick, blue!80!black,dashed] (-.25,-0.25) to (0.25,-0.25);
\draw[very thick, blue!80!black] (-.25,0) to (-.25,-0.25);
\draw[very thick, blue!80!black] (.25,0) to (.25,-0.25);
\draw[very thick, blue!80!black,dashed] (-.25,-1) to (0.25,-0.75);
\draw[very thick, blue!80!black,dashed] (-.25,-0.75) to (0.25,-1.0);
\draw[very thick, blue!80!black] (-.25,-1) to (-.25,-0.75);
\draw[very thick, blue!80!black] (.25,-1) to (.25,-0.75);
\draw[very thick, blue!80!black,dashed] (.35,-1.0) arc [start angle=-90, end angle=90, x radius=0.5,y radius =0.5] ;
\draw[very thick, blue!80!black,dashed] (.35,-0.75) arc [start angle=-90, end angle=90, x radius=0.25,y radius =0.25] ;
\draw[very thick, blue!80!black] (.35,0) to (.35,-0.25);
\draw[very thick, blue!80!black] (.35,-0.75) to (.35,-1.0);
\end{tikzpicture}
\caption{The equality of bordisms responsible for the unoriented Cardy condition \eqref{eq:unoriCardy}. All boundaries are labelled by the object $V \in \Brane$.}
\label{fig:unoriCardy}
\end{figure}
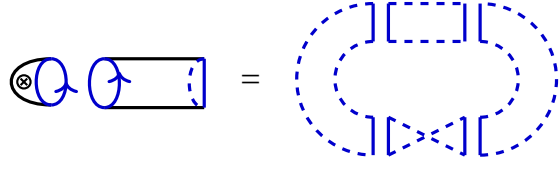

\begin{Prop}
\label{prop:orientifoldTFT}
Let $(\Brane, \bb^{\bullet},\bb_{\bullet},A)$ define a two dimensional oriented open/closed TFT and $(p,Q)$ an unoriented lift of $A$. Let $(P,\Theta)$ be a duality structure on $\Brane$ such that $\langle P(\phi) \rangle_{P(V)} = \langle \phi \rangle_V$ for all $\phi \in \Hom_{\Brane}(V,V)$ and $P \circ \bb_V = \bb_{P(V)} \circ p$ and $p \circ \bb^V = \bb^{P(V)} \circ P$ for all $V \in \Brane$. If the equality
\begin{equation}
\label{eq:LefInd}
\tr_{\Hom_{\Brane}(V,P(V))} \, \iota_{\phi} = \langle \bb^V(\phi),Q \rangle_0
\end{equation}
holds for all $\phi \in \Hom_{\Brane}(V,V)$, then $(\Brane^{\tilde{h}\Ctwo}, \bb^{\bullet},\bb_{\bullet},A,p,Q)$ defines a two dimensional unoriented open/closed TFT.
\end{Prop}

\begin{proof}
By Lemma \ref{lem:homoFixDual}, the triple $(\Brane^{\tilde{h}\Ctwo},P^{\tilde{h}\Ctwo},\Theta^{\tilde{h}\Ctwo})$ is a category with strict duality and $P^{\tilde{h}\Ctwo}$ acts trivially on objects. The category $\Brane^{\tilde{h}\Ctwo}$ inherits a Calabi--Yau structure from $\Brane$ with traces $\langle - \rangle_{(V,\psi_V)} := \langle- \rangle_V$. Define boundary-bulk and bulk-boundary maps for $\Brane^{\tilde{h}\Ctwo}$ by $\bb^{(V,\psi_V)} = \bb^V$ and $\bb_{(V,\psi_V)} = \bb_V$. The assumption that $P$ preserves the Calabi--Yau structure and that $P$ and $p$ are compatible with $\bb_{\bullet}$ and $\bb^{\bullet}$ verifies conditions \ref{ite:compatPairing}-\ref{ite:compatBoundBulk} of Theorem \ref{thm:genStruAlg} for $P^{\tilde{h} \Ctwo}$. It remains to verify the unoriented Cardy condition. Let $(V,\psi_V) \in \Brane^{\tilde{h}\Ctwo}$. Let $\{\psi_i\}_i$ be a basis of $\Hom_{\Brane}(V,P(V))$ with dual basis $\{\psi^i\}_{i}$ of $\Hom_{\Brane}(P(V),V)$. Then $\{\psi_V^{-1} \circ \psi_i \}_i$ is a basis of $\Hom_{\Brane}(V,V)$ with dual basis $\{\psi^i \circ \psi_V\}_{i}$. We compute
\begin{multline*}
\bb_{(V,\psi_V)}(Q)
=
\sum_i \psi^i \circ P(\psi_i) \circ \Theta_V = \\
\sum_i \psi^i \circ \psi_V \circ \psi^{-1}_V \circ P(\psi_i) \circ P(\psi_V^{-1}) \circ \psi_V
=
\sum_i \psi^i \circ \psi_V \circ P^{\tilde{h} \Ctwo}(\psi_V^{-1} \circ \psi_i).
\end{multline*}
For the first equality, note that the discussion proceeding Proposition \ref{prop:wittenIndices} shows that equation \eqref{eq:LefInd} implies that $\bb_V(Q) = \sum_i \psi^i \circ P(\psi_i) \circ \Theta_V$. The second equality follows from the coherence condition on homotopy fixed points and the final equality from the definition of $(P^{\tilde{h} \Ctwo},\Theta^{\tilde{h} \Ctwo})$.
\end{proof}

We comment on the physical interpretation of Proposition \ref{prop:orientifoldTFT}. As mentioned in the introduction, the Calabi--Yau category $\Brane$ should be seen as a model for the category of D-branes in an oriented string theory. With this interpretation, a duality structure $(P,\Theta)$ which preserves the Calabi--Yau pairings is the categorical data of the orientifold construction; see \cite{diaconescu2007,hori2008} in the setting of orientifolds of IIB string theory and Landau--Ginzburg theory. In this context, the quantity $\tr_{\Hom_{\Brane}(V,P(V))} \, \iota_{\phi}$ is a \emph{parity-twisted Witten index} \cite[\S 2]{brunner2004} and it is through its computation via closed sector quantities, namely equation \eqref{eq:LefInd}, that the crosscap state $Q$ naturally appears. The D-branes which survive the orientifold projection are the homotopy fixed points of $(P,\Theta)$, that is, objects of the category $\Brane^{\tilde{h}\Ctwo}$ above. With these remarks in mind, Proposition \ref{prop:orientifoldTFT} is an orientifold-type construction of an unoriented open/closed TFT from an oriented open/closed TFT.

\subsection{The Frobenius--Schur element as a crosscap state}

We give an algebraic construction of a two dimensional unoriented open/closed TFT from twisted Real representation theory. When $\Gh = \G \times \Ctwo$ and the cohomological data $(\hat{\theta},\lambda)$ is trivial, this generalizes results of \cite{alexeevski2006,loktev2011}. When $\lambda$ is trivial, a topological construction of the closed sector of this theory was given in \cite[\S 4.4]{mbyoung2020}.

Fix group theoretic data $(\Gh, \hat{\theta},\lambda)$ as in Section \ref{sec:RealProjRep}. Let $A = Z(\C^{\theta^{-1}}[\G])$ with Frobenius pairing $\langle - , - \rangle_G$ and $\Brane = \Rep^{\theta}(\G)$ the Calabi--Yau category with traces $\langle \phi \rangle_V = \frac{1}{\vert \G \vert} \tr_V \, \phi$. The boundary-bulk map $\bb^V$ is as in Section \ref{sec:FSIndDetail} and the bulk-boundary map is defined by
\[
\bb_V \Big(\sum_{g \in \G} a_g l_g \Big)
=
\sum_{g \in \G} a_g \theta([g \vert g^{-1}])^{-1} \rho_V(g^{-1}).
\]
This data defines a two dimensional oriented open/closed TFT $\TFT_{(\G,\theta)}$ via Theorem \ref{thm:ocOriTFT}. See \cite{moore2006,turaev2007,khoi2011}. The main axiom to be verified is the oriented Cardy condition which, in the present setting, is a mild generalization of the orthogonality of characters of irreducible $\theta$-twisted representations.

\begin{Thm}
\label{thm:twistRealRepTFT}
The data $(\Gh,\hat{\theta},\lambda)$ defines a two dimensional unoriented open/closed TFT $\TFT_{(\Gh,\hat{\theta},\lambda)}$ whose oriented sector is a sub TFT of $\TFT_{(\G,\theta)}$. 
\end{Thm}

We will prove Theorem \ref{thm:twistRealRepTFT} using the orientifold construction of Proposition \ref{prop:orientifoldTFT}. We take $(P^{(\hat{\theta},\lambda)},\Theta^{(\hat{\theta},\lambda)})$ for the duality structure on $\Brane = \Rep^{\theta}(\G)$ and $Q=\nu_{(\hat{\theta},\lambda)}$ for the candidate crosscap state. We compute
\[
\langle P(\phi) \rangle_{P(V)}
=
\frac{1}{\vert G \vert} \tr_{P(V)} \, P(\phi)
=
\frac{1}{\vert G \vert} \tr_{V^{\vee}} \, \phi^{\vee}
=
\langle \phi \rangle_V,
\]
which verifies the open sector assumption of Proposition \ref{prop:orientifoldTFT}. The remainder of the proof of Theorem \ref{thm:twistRealRepTFT} is divided into closed sector computations and verification of the open/closed coherence conditions required to apply Proposition \ref{prop:orientifoldTFT}.

The oriented open sector of $\TFT_{(\Gh,\hat{\theta},\lambda)}$ is the full theory $\TFT_{(\G,\theta)}$ precisely when the forgetful functor $\Rep^{\theta}(\G)^{\tilde{h}\Ctwo} \rightarrow \Rep^{\theta}(\G)$ is essentially surjective. By Corollary \ref{cor:FSInd}, this is the case when $\langle \chi_V, \nu_{(\hat{\theta},\lambda)} \rangle_{\G}=1$ for each irreducible $V \in \Rep^{\theta}(\G)$. Otherwise, the oriented open sector is a strict subtheory of $\TFT_{(\G,\theta)}$. In the context of Example \ref{ex:cyclicGroups}, the forgetful functor is essentially surjective only in subexample \eqref{ite:dihedralRealStr}.

\begin{Rem}
\label{rem:TypeIITwists}
We comment on the relation of $\TFT_{(\Gh, \hat{\theta},\lambda)}$ to categories of $D$-branes in orientifold string theory on global quotients. Recall that the spacetime of an orientifold string theory is an orbifold double cover $\pi: \mathcal{X} \rightarrow \hat{\mathcal{X}}$. Additional data required to define the theory includes the (gauge equivalence class of a) $B$-field $\check{B} \in \check{H}^{3 + \pi}(\hat{\mathcal{X}})$ \cite{distler2011b}, which is a class in the $\pi$-twisted differential cohomology of $\hat{\mathcal{X}}$, and a complex line bundle with connection $\check{L} \in \check{H}^2(\hat{\mathcal{X}})$ \cite[\S 8.4.1]{gao2011}. The underlying (oriented) orbifold string theory depends only on $(\mathcal{X}, \pi^* \check{B})$. Consider now the particular case in which the spacetime is a global quotient $\pi: X \git \G \rightarrow X \git \Gh$ associated to a finite $\Ctwo$-graded group $\Gh$ acting on a smooth manifold $X$. A special class of $B$-fields arises through the composition
\[
H^{2+\pi}(B \Gh) \rightarrow H^{2+\pi}(X \git \Gh) \hookrightarrow \check{H}^{3+\pi}(X \git \G),
\qquad
\hat{\theta} \mapsto \check{B}_{\hat{\theta}},
\]
where the first map is restriction along the the canonical morphism $X \git \Gh \rightarrow B \Gh$ and the second is the inclusion of flat $B$-fields. Similarly, a class $\lambda \in H^1(B \Gh)$ defines a flat line bundle $\check{L}_{\lambda} \in \check{H}^2(X \git \Gh)$. The pair $(\check{B}_{\hat{\theta}},\check{L}_{\lambda})$ can be seen as defining universal twists for global $\Gh$-orientifolds. The unoriented TFT $\TFT_{(\Gh, \hat{\theta},\lambda)}$ is a precise mathematical description of the affects of the twists $(\check{B}_{\hat{\theta}},\check{L}_{\lambda})$ on partition functions. See \cite{braun2002}, \cite[\S 5]{sharpe2011}, \cite[\S 4.5]{noohiYoung2022} for detailed discussions of these affects in the closed sector.
\end{Rem}

We return to the proof of Theorem \ref{thm:twistRealRepTFT}.

\subsubsection{Closed sector}

Denote by $\Aut^{\textnormal{gen}}(\C^{\theta^{-1}}[G])$ the group of algebra automorphisms and algebra anti-automorphisms of $\C^{\theta^{-1}}[G]$. The group $\Aut^{\textnormal{gen}}(\C^{\theta^{-1}}[G])$ is $\Ctwo$-graded by sending anti-automorphisms to $-1$.

\begin{Lem}
\label{lem:genGAct}
The function $p: \Gh \rightarrow \Aut^{\textnormal{gen}}(\C^{\theta^{-1}}[\G])$, $\omega \mapsto p^{\omega}$, where
\[
p^{\omega}(l_g) = \lambda(g)^{\frac{\pi(\omega)-1}{2}}\uptau_{\pi}^{\refl}(\hat{\theta})([\omega] g) l_{\omega g^{\pi(\omega)} \omega^{-1}},
\qquad
g \in \G
\]
is a $\Ctwo$-graded group homomorphism. Moreover, each $p^{\omega}$ is an isometry of $\langle - \rangle_G$.
\end{Lem}

\begin{proof}
We prove that $p^{\omega}$, $\omega \in \Gh \backslash \G$, is an anti-automorphism and omit the easier calculation that $p^g$, $g \in \G$, is an automorphism. For $g,h \in \G$, direct calculations give
\[
p^{\omega}(l_g \cdot l_h)
=
\frac{\uptau_{\pi}^{\refl}(\hat{\theta})([\omega]gh)}{\lambda(gh) \theta([g \vert h])} l_{\omega (gh)^{-1} \omega^{-1}}
\]
and
\[
p^{\omega}(l_h) \cdot p^{\omega}(l_g)
=
\frac{\uptau_{\pi}^{\refl}(\hat{\theta})([\omega]h) \uptau_{\pi}^{\refl}(\hat{\theta})([\omega]g)}{\lambda(g) \lambda(h) \theta([\omega h^{-1} \omega^{-1} \vert \omega g^{-1} \omega^{-1}])} l_{\omega  h^{-1} g^{-1} \omega^{-1}}.
\]
It therefore suffices to prove that
\[
\frac{\theta([g \vert h])}{\theta([\omega h^{-1} \omega^{-1} \vert \omega g^{-1} \omega^{-1}])}
=
\frac{\uptau_{\pi}^{\refl}(\hat{\theta})([\omega]gh)}{\uptau_{\pi}^{\refl}(\hat{\theta})([\omega]h) \uptau_{\pi}^{\refl}(\hat{\theta})([\omega]g)}.
\]
A short calculation using equation \eqref{eq:2cocycleKey} shows that this identity indeed holds.

That $p^{\omega}$ is an isometry follows from the equalities
\[
\langle p^{\omega}(l_g) \rangle_{\G}
=
\frac{1}{\vert \G \vert} \delta_{e,g} \frac{\uptau_{\pi}^{\refl}(\hat{\theta})}{\lambda(e)}([\omega] e)
=
\frac{\delta_{e,g}}{\vert \G \vert}
=
\langle l_g \rangle_{\G}.
\]

It remains to prove the homomorphism property, $p^{\omega_2} \circ p^{\omega_1} = p^{\omega_2 \omega_1}$. Recall from Section \ref{sec:RealProjRep} that $\uptau_{\pi}^{\refl}(\hat{\theta})$ is a $1$-cocycle on the groupoid $\G \git_R \Gh$ whose objects are elements of $\G$ and whose morphisms are $\omega: g \rightarrow \omega g^{\pi(\omega)} \omega^{-1}$, $\omega \in \Gh$. With this description, closedness of $\uptau_{\pi}^{\refl}(\hat{\theta})$ becomes the equalities
\[
\uptau_{\pi}^{\refl}(\hat{\theta})([\omega_2]\omega_1 g^{\pi(\omega_1)} \omega_1^{-1}) \uptau_{\pi}^{\refl}(\hat{\theta})([\omega_1]g) = \uptau_{\pi}^{\refl}(\hat{\theta})([\omega_2 \omega_1]g),
\qquad
g \in \G, \; \omega_i \in \Gh
\]
which are immediately seen to imply the homomorphism property.
\end{proof}

\begin{Prop}
\label{prop:algInvolution}
For each $\varsigma \in \Gh \backslash \G$, the map $p^{\varsigma}$ restricts to an algebra involution of $Z(\C^{\theta^{-1}}[\G])$. Moreover, this involution is independent of $\varsigma$.
\end{Prop}

\begin{proof}
Using the explicit descriptions of the centre $Z(\C^{\theta^{-1}}[\G])$ from Section \ref{sec:twistRepThy} and the $\G$-action on $\C^{\theta^{-1}}[\G]$ from Lemma \ref{lem:genGAct} we see that $\C^{\theta^{-1}}[\G]^{\G}=Z(\C^{\theta^{-1}}[\G])$. It follows that the generalized $\Gh$-action on $\C^{\theta^{-1}}[\G]$ from Lemma \ref{lem:genGAct} induces an action of $\Ctwo \simeq \Gh \slash \G$ by algebra automorphisms on $Z(\C^{\theta^{-1}}[\G])$.
\end{proof}

Denote by $p$ the algebra involution of $Z(\C^{\theta^{-1}}[\G])$ induced by any $\varsigma \in \Gh \backslash \G$.

\begin{Rem}
Using functoriality of Hochschild homology and invariance under taking opposites, we form the composition
\begin{equation}
\label{eq:HHInv}
HH_{\bullet}(\Rep^{\theta}(\G)) \xrightarrow[]{\sim} HH_{\bullet}(\Rep^{\theta}(\G)^{\op}) \xrightarrow[]{HH_{\bullet}(P^{(\hat{\theta},\lambda)})} HH_{\bullet}(\Rep^{\theta}(\G)).
\end{equation}
Since $\Rep^{\theta}(\G)$ is finite semisimple, $HH_{\bullet}(\Rep^{\theta}(\G))$ is concentrated in degree zero, where is it isomorphic to $Z(\C^{\theta^{-1}}[\G])$. Under this isomorphism, the map \eqref{eq:HHInv} is $p$. The $+1$ (resp. $-1$) eigenspace of $p$ is then the involutive (resp. skew-involutive) Hochschild homology of $(\Rep^{\theta}(\G), P^{(\hat{\theta},\lambda)},\Theta^{(\hat{\theta},\lambda)})$. See \cite[Theorem 2.14]{braun2014} for an analogous result in the setting of strictly involutive $A_{\infty}$-algebras. The first part of Proposition \ref{prop:wittenIndices} therefore shows that the Klein condition computes the difference in dimensions of involutive and skew-involutive Hochschild homologies.
\end{Rem}

\begin{Prop}
\label{prop:crosscapInv}
The element $\nu_{(\hat{\theta},\lambda)} \in Z(\C^{\theta^{-1}}[\G])$ is $p$-invariant.
\end{Prop}

\begin{proof}
We have seen in Lemma \ref{lem:FSClassFun} that $\nu_{(\hat{\theta},\lambda)} \in Z(\C^{\theta^{-1}}[\G])$. For $p$-invariance, we have
\[
p^{\varsigma}(\nu_{(\hat{\theta},\lambda)})
=
\sum_{\mu \in \Gh \backslash \G}  \frac{\lambda(\mu) \uptau_{\pi}^{\refl}(\hat{\theta})([\varsigma] \mu^2)}{\lambda(\mu^2) \hat{\theta}([\mu \vert \mu])} l_{\varsigma \mu^{-2} \varsigma^{-1}}.
\]
Equation \eqref{eq:oddConj} gives
$
\hat{\theta}([\varsigma \mu^{-1} \varsigma^{-1} \vert \varsigma \mu^{-1} \varsigma^{-1}]) = \frac{\hat{\theta}([\varsigma \vert \mu^{-2}])}{\hat{\theta}([\mu^{-1} \vert \mu^{-1}]) \hat{\theta}([\varsigma \mu^{-2} \varsigma^{-1} \vert \varsigma])}
$
so that $p^{\varsigma}(\nu_{(\hat{\theta},\lambda)})$ is equal to
\begin{multline*}
\sum_{\mu \in \Gh \backslash \G} \lambda(\mu)^{-1} 
\frac{\hat{\theta}([\varsigma \mu^{-2} \varsigma^{-1} \vert \varsigma])}{\hat{\theta}([\mu \vert \mu]) \hat{\theta}([\mu^2 \vert \mu^{-2}]) \hat{\theta}([\varsigma \vert \mu^{-2}])} l_{\varsigma \mu^{-2} \varsigma^{-1}} \\
=
\sum_{\mu \in \Gh \backslash \G} \lambda(\mu)^{-1} \hat{\theta}([\mu \vert \mu])^{-1} \hat{\theta}([\mu^2 \vert \mu^{-2}])^{-1} \hat{\theta}([\mu^{-1} \vert \mu^{-1}])^{-1} \hat{\theta}([\varsigma \mu^{-1} \varsigma^{-1} \vert \varsigma \mu^{-1} \varsigma^{-1}])^{-1} l_{\varsigma \mu^{-2} \varsigma^{-1}}.
\end{multline*}
A short calculation shows that $\hat{\theta}([\mu^{-1} \vert \mu^{-1}]) \hat{\theta}([\mu \vert \mu]) \hat{\theta}([\mu^2 \vert \mu^{-2}]) =1$, whence
\[
p^{\varsigma}(\nu_{(\hat{\theta},\lambda)}) =
\sum_{\mu \in \Gh \backslash \G} \lambda(\mu^{-1}) \hat{\theta}([\varsigma \mu^{-1} \varsigma^{-1} \vert \varsigma \mu^{-1} \varsigma^{-1}])^{-1} l_{\varsigma \mu^{-2} \varsigma^{-1}} =\nu_{(\hat{\theta},\lambda)}. \qedhere
\]
\end{proof}

\begin{Lem}
\label{lem:crosscapHelp}
The following equality holds for all $g \in \G$ and $\mu \in \Gh \backslash \G$:
\[
\hat{\theta}([\mu \vert \mu])^{-1} l_{\mu^2} \cdot a_g l_g = \lambda(g) p^{\mu}(a_g l_g) \cdot \hat{\theta}([\mu g \vert \mu g])^{-1} l_{(\mu g)^2}.
\]
\end{Lem}

\begin{proof}
This can be verified directly from the twisted $2$-cocycle condition on $\hat{\theta}$.
\end{proof}

\begin{Prop}
The equality $p(\nu_{(\hat{\theta},\lambda)} f) = \nu_{(\hat{\theta},\lambda)} f$ holds for all $f \in Z(\C^{\theta^{-1}}[\G])$.
\end{Prop}

\begin{proof}
Write $\sum_{g \in \G} a_g l_g$ for $f \in Z(\C^{\theta^{-1}}[\G])$. Lemma \ref{lem:crosscapHelp} gives
\[
\nu_{(\hat{\theta},\lambda)} \sum_{g \in \G} a_g l_g
=
\sum_{\substack{g \in \G \\ \mu \in \Gh \backslash \G}} \lambda(\mu g) p^{\mu}(a_g l_g) \hat{\theta}([\mu g \vert \mu g])^{-1} l_{(\mu g)^2}
\]
from which we find that $p^{\varsigma}(\nu_{(\hat{\theta},\lambda)} \sum_g a_g l_g)$ is equal to
\begin{eqnarray*}
&&
\sum_{\substack{g \in \G \\ \mu \in \Gh \backslash \G}} \lambda(\mu g) p^{\varsigma} (p^{\mu}(a_g l_g) \cdot \hat{\theta}([\mu g \vert \mu g])^{-1} l_{(\mu g)^2}) \\
&=&
\sum_{g, \mu} \lambda(\mu g) p^{\varsigma}(\hat{\theta}([\mu g \vert \mu g])^{-1} l_{(\mu g)^2}) p^{\varsigma} p^{\mu}(a_g l_g) \\
&=&
\sum_{g, \mu} \lambda(\mu g)^{-1} \hat{\theta}([\varsigma g^{-1} \mu^{-1} \varsigma^{-1} \vert \varsigma g^{-1} \mu^{-1} \varsigma^{-1}])^{-1} l_{\varsigma (g^{-1} \mu^{-1})^2 \varsigma^{-1}} p^{\varsigma} p^{\mu}(a_g l_g) \\
&=&
\sum_{g, \mu} \lambda(\mu g) \hat{\theta}([\varsigma g^{-1} \mu^{-1} \varsigma^{-1} \vert \varsigma g^{-1} \mu^{-1} \varsigma^{-1}])^{-1} l_{\varsigma (g^{-1} \mu^{-1})^2 \varsigma^{-1}} \uptau_{\pi}^{\refl}(\hat{\theta})([\varsigma \mu]g)^{-1} a_g l_{\varsigma \mu g (\varsigma \mu)^{-1}} \\
&=&
\sum_{g, \mu} \lambda(\mu g) \hat{\theta}([\varsigma g^{-1} \mu^{-1} \varsigma^{-1} \vert \varsigma g^{-1} \mu^{-1} \varsigma^{-1}])^{-1} l_{\varsigma (g^{-1} \mu^{-1})^2 \varsigma^{-1}} a_{\varsigma \mu g (\varsigma \mu)^{-1}} l_{\varsigma \mu g (\varsigma \mu)^{-1}} \\
&=&
\sum_{\substack{h \in \G \\ \eta \in \Gh \backslash \G}} \frac{\lambda(\eta)}{\hat{\theta}([ \eta \vert \eta])} l_{\eta^2} a_h l_h,
\end{eqnarray*}
which is $\nu_{(\hat{\theta},\lambda)} \sum_{h \in \G} a_h l_h$. The first equality follows from the fact that $p^{\varsigma}$ is an anti-homomorphism (Lemma \ref{lem:genGAct}), the second from Proposition \ref{prop:crosscapInv}, the third from Lemma \ref{lem:genGAct} and the definition of $p$, the fourth from the assumed centrality of $\sum_g a_g l_g$ and the fifth from the change of variables $\eta = \varsigma g^{-1} \mu^{-1} \varsigma^{-1}$ and $h = \varsigma \mu g \mu^{-1} \varsigma^{-1}$.
\end{proof}

Recall that a conjugacy class $\mathcal{O} \subset \G$ is called \emph{$\theta$-regular} if $\frac{\theta([g \vert h])}{\theta([h \vert g])} =1$ for all $g \in \mathcal{O}$ and $h \in C_{\G}(g)$.

\begin{Prop}
The Klein condition holds.
\end{Prop}

\begin{proof}
The vector space $Z(\C^{\theta^{-1}}[\G])$ has a basis $\{ l_{\mathcal{O}} \}_{\mathcal{O}}$ labelled by $\theta$-regular conjugacy classes of $\G$. For convenience, set $l_{\mathcal{O}}=0$ if $\mathcal{O}$ is not $\theta$-regular. Writing $l_{\mathcal{O}} = \sum_{g \in \mathcal{O}} a_g l_g$ and $l_{\mathcal{O}^{-1}} = \sum_{h \in \mathcal{O}} b_{h^{-1}} l_{h^{-1}}$, we have
$
\langle l_{\mathcal{O}}, l_{\mathcal{O}^{-1}} \rangle_{\G}
=
\frac{1}{\vert \G \vert} \sum_{g \in \mathcal{O}}\theta([g \vert g^{-1}])^{-1} a_g b_{g^{-1}}.
$
Centrality of $l_{\mathcal{O}^{\pm 1}}$ implies that the function $\mathcal{O} \rightarrow \C$, $g \mapsto \theta([g \vert g^{-1}])^{-1} a_g b_{g^{-1}}$, is constant; denote its (necessarily non-zero) value by $c_{\mathcal{O}}$. We also have $\langle l_{\mathcal{O}}, l_{\mathcal{O}^{\prime}} \rangle_{\G} = 0$ if $\mathcal{O}^{\prime} \neq \mathcal{O}^{-1}$. It follows that $l_{\mathcal{O}}^{\vee} = \frac{\vert \G \vert}{c_{\mathcal{O}} \vert \mathcal{O} \vert} l_{\mathcal{O}^{-1}}$. With this notation, the right hand side of the Klein condition is $R := \sum_{\mathcal{O} \in \pi_0(\G \git \G)} l_{\mathcal{O}} p^{\varsigma}(l_{\mathcal{O}}^{\vee})$. We compute
\begin{eqnarray*}
R
&=&
\vert \G \vert \sum_{\mathcal{O} \in \pi_0(\G \git \G)} \sum_{g, h \in \mathcal{O}} \frac{a_g l_g  p^{\varsigma}(b_{h^{-1}} l_{h^{-1}})}{ c_{\mathcal{O}} \vert \mathcal{O} \vert} \\
&=&
\vert \G \vert \sum_{\mathcal{O} \in \pi_0(\G \git \G)} \sum_{g, h \in \mathcal{O}} \frac{a_g b_{h^{-1}}\lambda(h) \uptau_{\pi}^{\refl}(\hat{\theta})([\varsigma] h^{-1}) l_{g\varsigma h \varsigma^{-1}}}{c_{\mathcal{O}} \vert \mathcal{O} \vert \hat{\theta}([g \vert \varsigma h \varsigma^{-1}])} \\
&=&
\sum_{\mathcal{O} \in \pi_0(\G \git \G)} \sum_{g \in \mathcal{O}} \sum_{t \in \G} \frac{a_g b_{t g^{-1} t^{-1}}\lambda(g) \uptau_{\pi}^{\refl}(\hat{\theta})([\varsigma] t g^{-1} t^{-1}) l_{g\varsigma t g t^{-1} \varsigma^{-1}}}{c_{\mathcal{O}} \hat{\theta}([g \vert \varsigma t g t^{-1} \varsigma^{-1}])} \\
&=&
\sum_{g, t \in \G} \frac{a_g b_{t g^{-1} t^{-1}}}{ c_{\mathcal{O}}} \lambda(g) \frac{\uptau_{\pi}^{\refl}(\hat{\theta})([\varsigma] t g^{-1} t^{-1})}{\hat{\theta}([g \vert \varsigma t g t^{-1} \varsigma^{-1}])} l_{g\varsigma t g t^{-1} \varsigma^{-1}}.
\end{eqnarray*}
Above we have set $h = t g t^{-1}$. As $b_{t g^{-1} t^{-1}} = \uptau(\theta)([t]g^{-1}) b_{g^{-1}}$, we can write
\begin{eqnarray*}
R
&=&
\sum_{g, t \in \G} \frac{a_g b_{g^{-1}}}{ c_{\mathcal{O}}} \lambda(g) \frac{\uptau(\theta)([t]g^{-1}) \uptau_{\pi}^{\refl}(\hat{\theta})([\varsigma] t g^{-1} t^{-1})}{\hat{\theta}([g \vert \varsigma t g t^{-1} \varsigma^{-1}])} l_{g\varsigma t g^{-1} t^{-1} \varsigma^{-1}} \\
&=&
\sum_{g, t \in \G} \lambda(g) \frac{\theta([g \vert g^{-1}])}{\hat{\theta}([g \vert \varsigma t g t^{-1} \varsigma^{-1}])} \uptau(\theta)([t]g^{-1}) \uptau_{\pi}^{\refl}(\hat{\theta})([\varsigma] t g^{-1} t^{-1}) l_{g\varsigma t g t^{-1} \varsigma^{-1}} \\
&=&
\sum_{g, t \in \G} \lambda(g) \frac{\theta([g \vert g^{-1}])}{\hat{\theta}([g \vert \varsigma t g t^{-1} \varsigma^{-1}])} \uptau_{\pi}^{\refl}(\hat{\theta})([\varsigma t] g^{-1}) l_{g\varsigma t g t^{-1} \varsigma^{-1}}.
\end{eqnarray*}
The second equality follows from the definition of $c_{\mathcal{O}}$ and the final from closedness of $\uptau_{\pi}^{\refl}(\hat{\theta})$, as in the proof of Lemma \ref{lem:genGAct}. Define $\mu, \xi \in \Gh \backslash \G$ by $\mu = g \varsigma t$ and $\xi = t^{-1} \varsigma^{-1}$, so that $t = \varsigma^{-1} \xi^{-1}$ and $g = \mu \xi$. Making these substitutions, the coefficient of $l_{g\varsigma t g t^{-1} \varsigma^{-1}} = l_{\mu^2 \xi^2}$ in $R$ is
\[
\lambda(\mu \xi) \hat{\theta}([\mu \xi \vert \xi^{-1} \mu \xi^2])^{-1} \frac{\hat{\theta}([\xi^{-1} \mu \xi^2 \vert \xi^{-1}])}{\hat{\theta}([\xi^{-1} \vert \mu \xi])}
=
\lambda(\mu \xi) \hat{\theta}([\mu \vert \mu])^{-1} \hat{\theta}([\xi \vert \xi])^{-1} \hat{\theta}([\mu^2 \vert \xi^2])^{-1}.
\]
It follows that
\[
R =\sum_{\mu, \xi \in \Gh \backslash \G} \lambda(\mu \xi) \hat{\theta}([\mu \vert \mu])^{-1} \hat{\theta}([\xi \vert \xi])^{-1} \hat{\theta}([\mu^2 \vert \xi^2])^{-1} l_{\mu^2 \xi^2},
\]
which is exactly $\nu_{(\hat{\theta},\lambda)}^2$.
\end{proof}

\subsubsection{Open/closed coherence}

Note that Theorem \ref{thm:LefGroupAlg} verifies equation \eqref{eq:LefInd}.

\begin{Prop}
The maps $\bb^{\bullet}$, $\bb_{\bullet}$, $P$ and $p$ satisfy the assumptions of Proposition \ref{prop:orientifoldTFT}.
\end{Prop}

\begin{proof}
We compute
\[
P \circ \bb_V (p(\sum_{g \in \G} a_g l_g))
=
\sum_{g \in \G} a_g \frac{\uptau_{\pi}^{\refl}(\hat{\theta})([\varsigma]g)}{\lambda(g)} \rho_V(\varsigma g^{-1} \varsigma^{-1})^{\vee}
=
\bb_{P(V)}(\sum_{g \in \G} a_g l_g),
\]
that is, $P \circ \bb_V \circ p = \bb_{P(V)}$. Since $p$ is an involution, this implies $P \circ \bb_V = \bb_{P(V)} \circ p$.

We also have
\begin{eqnarray*}
p \circ \bb^V(\phi)
&=&
\sum_{g \in \G} \frac{\uptau_{\pi}^{\refl}(\hat{\theta})([\varsigma] g)}{\lambda(g)} \tr_V (\phi \circ \rho_V(g^{-1}))  l_{\varsigma g^{-1} \varsigma^{-1}}
\end{eqnarray*}
and
\begin{eqnarray*}
\bb^{P(V)} ( P(\phi))
&=&
\sum_{g \in \G} \frac{\lambda(g)}{\uptau_{\pi}^{\refl}(\hat{\theta})([\varsigma]g)} \tr_{V^{\vee}} (\phi^{\vee} \circ \rho_V(\varsigma g^{-1} \varsigma^{-1})^{\vee}) l_g \\
&=&
\sum_{g \in \G} \frac{\lambda(\varsigma g^{-1} \varsigma^{-1})}{\uptau_{\pi}^{\refl}(\hat{\theta})(\varsigma g^{-1} \varsigma^{-1})} \tr_{V^{\vee}} (\phi^{\vee} \circ \rho_V(\varsigma^2 g \varsigma^{-2})^{\vee}) l_{\varsigma g^{-1} \varsigma^{-1}}.
\end{eqnarray*}
Closedness of $\uptau_{\pi}^{\refl}(\hat{\theta})$ implies
\[
\uptau_{\pi}^{\refl}(\hat{\theta})([\varsigma] g)
\uptau_{\pi}^{\refl}(\hat{\theta})([\varsigma] \varsigma g^{-1} \varsigma^{-1}) = \uptau_{\pi}^{\refl}(\hat{\theta})([\varsigma^2] g).
\]
Since $\phi$ is $\G$-equivariant and 
\[
\rho_V(\varsigma^2) \circ \rho_V(g) \circ \rho_V(\varsigma^{-2})
=
\theta([\varsigma^2 \vert g]) \theta([\varsigma^2 g \vert \varsigma^{-2}])
\rho_V(\varsigma^2 g  \varsigma^{-2}),
\]
we have
\[
\tr_{V^{\vee}} (\phi^{\vee} \circ \rho_V(\varsigma^2 g  \varsigma^{-2})^{\vee})
=
\frac{\theta([\varsigma^{-2} \vert \varsigma^2])}{\theta([\varsigma^2 \vert g]) \theta([\varsigma^2 g \vert \varsigma^{-2}])} \tr_V (\phi \circ \rho_V(g)).
\]
The coefficient of $\tr_V (\phi \circ \rho_V(g)) l_{\varsigma g^{-1} \varsigma^{-1}}$ in $\bb^{P(V)} (P(\phi))$ is therefore
\[
\lambda(g)^{-1} \frac{\uptau_{\pi}^{\refl}(\hat{\theta})([\varsigma] g)}{\uptau_{\pi}^{\refl}(\hat{\theta})([\varsigma^2] g)} \frac{\theta([\varsigma^{-2} \vert \varsigma^2])}{\theta([\varsigma^2 \vert g]) \theta([\varsigma^2 g \vert \varsigma^{-2}])}
=
\frac{\uptau_{\pi}^{\refl}(\hat{\theta})([\varsigma] g)}{\lambda(g)}.
\]
We conclude that $p \circ \bb^V = \bb^{P(V)} \circ P$.
\end{proof}

This completes the proof of Theorem \ref{thm:twistRealRepTFT}.

\subsection{Partition functions}
\label{sec:partFun}

The algebraic construction of $\TFT_{(\Gh,\hat{\theta},\lambda)}$ allows for the explicit computation of the partition function of an arbitrary surface.

\subsubsection{Closed surfaces}

For the real projective plane, we have
\[
\TFT_{(\Gh,\hat{\theta},\lambda)}(\mathbb{RP}^2)
=
\langle \nu_{(\hat{\theta},\lambda)} \rangle_{\G} =
\frac{1}{\vert \G \vert} \sum_{\substack{\mu \in \Gh \backslash \G \\ \mu^2 = e }} \frac{\lambda(\mu)}{\hat{\theta}([\mu \vert \mu])},
\]
the first equality reflecting that $\mathbb{RP}^2$ is a M\"{o}bius strip glued to a disk. In particular, $\TFT_{(\Gh,\hat{\theta},\lambda)}(\mathbb{RP}^2)$ vanishes unless $\pi: \Gh \rightarrow \Ctwo$ splits. Realizing the Klein bottle as two cylinders glued together, with one gluing by circle reflection, and using that $Z(\C^{\theta^{-1}}[\G]) = \C^{\theta^{-1}}[\G]^{\G}$ (see the proof of Proposition \ref{prop:algInvolution}), we compute
\[
\TFT_{(\Gh,\hat{\theta},\lambda)}(\mathbb{K})
=
\frac{1}{\vert \G \vert} \sum_{h \in \G} \tr_{\C^{\theta^{-1}}[\G]} p^{h\varsigma} =
\frac{1}{\vert \G \vert} \sum_{\substack{g \in \G \\ \omega \in \Gh \backslash \G \\ \omega g^{-1} \omega^{-1} = g}} \frac{1}{\lambda(g) \hat{\theta}([g^{-1} \vert g])} \frac{\hat{\theta}([g \vert \omega])}{\hat{\theta}([\omega \vert g^{-1}])}.
\]
In general, a formula for the partition function of a closed connected non-orientable surface $\Sigma$ can be written in terms of $\hat{\theta}$-weighted counts of $\Ctwo$-graded homomorphisms from the fundamental group of the orientation double cover of $\Sigma^{\ori} \rightarrow \Sigma$ to $\Gh$. See \cite[\S 4.4]{mbyoung2020}. Alternatively, the primitive orthogonal idempotents of the semisimple algebra $Z(\C^{\theta^{-1}}[\G])$ can be used to evaluate the partition functions. Proceeding in this way and writing the crosscap state as in Corollary \ref{cor:FSIndDecomp}, we find
\[
\TFT_{(\Gh,\hat{\theta},\lambda)}(\Sigma)
=
\sum_{\substack{ V \in \Irr^{\theta}(\G) \\ P^{(\hat{\theta},\lambda)}(V) \simeq V}} \left(\frac{\langle \chi_V, \nu_{(\hat{\theta},\lambda)} \rangle_G \dim_{\C} V}{\vert \G \vert} \right)^{\chi(\Sigma)}.
\]
For example,
\[
\TFT_{(\Gh,\hat{\theta},\lambda)}(\mathbb{RP}^2)
=
\frac{1}{\vert \G \vert}
\sum_{\substack{ V \in \Irr^{\theta}(\G) \\ P^{(\hat{\theta},\lambda)}(V) \simeq V}} \langle \chi_V, \nu_{(\hat{\theta},\lambda)} \rangle_G \dim_{\C} V
\]
and
\[
\TFT_{(\Gh,\hat{\theta},\lambda)}(\mathbb{K})
=
\vert \{ V \in \Irr^{\theta}(\G) \mid P^{(\hat{\theta},\lambda)}(V) \simeq V\} \vert.
\]
Equating these expressions for the partition function of $\Sigma$ relates weighted counts of $\Ctwo$-graded homomorphisms $\pi_1(\Sigma^{\ori}) \rightarrow \Gh$ to Real character theoretic sums. Various specializations of these identities are known \cite{frobenius1906,karimipour1997,mulase2005,snyder2017,barkeshli2020,mbyoung2020} and provide non-orientable counterparts of Mednykh's formulae \cite{mednykh1978}.

\subsubsection{Surfaces with boundary}

Let $\Sigma$ be a compact connected non-orientable surface with $b \geq 1$ boundary components. To begin, label each boundary component by the same irreducible twisted Real representation $V \in \Rep^{\theta}(\G)^{\tilde{h}\Ctwo}$. The partition function
\[
\TFT_{(\Gh,\hat{\theta},\lambda)}(\Sigma;V)
: \Hom_{\Rep^{\theta}(\G)^{\tilde{h}\Ctwo}}(V,V)^{\otimes b} \rightarrow \C
\]
can be computed as follows. By Proposition \ref{prop:RealRestr}, there are two cases to consider. If $V$ is irreducible as a twisted representation, then the primitive orthogonal idempotent of $\C^{\theta^{-1}}[\G]$ corresponding to $V$ is $e_{V^{\vee}} = \frac{\dim_{\C} V}{\vert \G \vert} \chi_{V}$, whence $\chi_V^b = \left(\frac{\dim_{\C} V}{\vert \G \vert}\right)^{-b} e_{V^{\vee}}$. Using this and the fact that $\langle \chi_V, \nu_{(\hat{\theta},\lambda))} \rangle_{\G} =1$ (see Corollary \ref{cor:FSInd}), we compute
\[
\TFT_{(\Gh,\hat{\theta},\lambda)}(\Sigma;V)(\id_V^{\otimes b})
=
\left(\frac{\dim_{\C} V}{\vert \G \vert} \right)^{\chi(\Sigma)}.
\]
If instead the underlying twisted representation of $V$ is reducible, then $V \simeq H^{(\hat{\theta},\lambda)}(U)$ with $U \in \Rep^{\theta}(\G)$ irreducible. It follows that $\chi_{H^{(\hat{\theta},\lambda)}(U)}=\frac{\vert \G \vert}{\dim_{\C} U} (e_{U^{\vee}} + e_{P^{(\hat{\theta},\lambda)}(U)^{\vee}})$ and
\[
\chi_{H^{(\hat{\theta},\lambda)}(U)}^{b}
=\left( \frac{\dim_{\C} U}{\vert \G \vert} \right)^{-b} (e_{U^{\vee}} + e_{P^{(\hat{\theta},\lambda)}(U)^{\vee}}).
\]
There are two further sub-cases:
\begin{itemize}
\item $P^{(\hat{\theta},\lambda)}(U) \not\simeq U$, in which case $\langle \chi_U, \nu_{(\hat{\theta},\lambda)}\rangle_{\G} = \langle \chi_{P^{(\hat{\theta},\lambda)}(U)} \nu_{(\hat{\theta},\lambda)}\rangle_{\G}
= 0$. Since $\nu_{(\hat{\theta},\lambda)}$ has no $\chi_U$ or $\chi_{P^{(\hat{\theta},\lambda)}(U)}$ components, we find
\[
\TFT_{(\Gh,\hat{\theta},\lambda)}(\Sigma;H^{(\hat{\theta},\lambda)}(U))(\id_{H^{(\hat{\theta},\lambda)}(U)}^{\otimes b})=0.
\]

\item $P^{(\hat{\theta},\lambda)}(U) \simeq U$, in which case $\langle \chi_U, \nu_{(\hat{\theta},\lambda)}\rangle_{\G}= -1$ and
\[
\TFT_{(\Gh,\hat{\theta},\lambda)}(\Sigma;H^{(\hat{\theta},\lambda)}(U))(\id_{H^{(\hat{\theta},\lambda)}(U)}^{\otimes b})
=
2\left(-\frac{\dim_{\C} U}{\vert \G \vert} \right)^{\chi(\Sigma)}.
\]
\end{itemize}

A formula in the general case, with boundary components labelled by arbitrary twisted Real representations $V_i$, $i=1, \dots, b$, can be deduced from the previous formulae by writing $\chi_{V_i}$ as a linear combination of primitive idempotents.

\bibliographystyle{amsalpha}

\begin{thebibliography}{DGRKS07}

\bibitem[Abo10]{abouzaid2010}
M.~Abouzaid, \emph{A geometric criterion for generating the {F}ukaya category},
  Publ. Math. Inst. Hautes \'Etudes Sci. (2010), no.~112, 191--240. \MR{2737980
  (2012h:53192)}

\bibitem[Abr96]{abrams1996}
L.~Abrams, \emph{Two-dimensional topological quantum field theories and
  {F}robenius algebras}, J. Knot Theory Ramifications \textbf{5} (1996), no.~5,
  569--587. \MR{1414088 (97j:81292)}

\bibitem[AN06]{alexeevski2006}
A.~Alexeevski and S.~Natanzon, \emph{Noncommutative two-dimensional topological
  field theories and {H}urwitz numbers for real algebraic curves}, Selecta
  Math. (N.S.) \textbf{12} (2006), no.~3-4, 307--377. \MR{2305607
  (2008a:57032)}

\bibitem[AS69]{atiyah1969}
M.~Atiyah and G.~Segal, \emph{Equivariant {$K$}-theory and completion}, J.
  Differential Geometry \textbf{3} (1969), 1--18. \MR{0259946}

\bibitem[BBC{\etalchar{+}}20]{barkeshli2020}
M.~Barkeshli, P.~Bonderson, M.~Cheng, C.-M. Jian, and K.~Walker,
  \emph{Reflection and time reversal symmetry enriched topological phases of
  matter: path integrals, non-orientable manifolds, and anomalies}, Comm. Math.
  Phys. \textbf{374} (2020), no.~2, 1021--1124.

\bibitem[BCT09]{blumberg2009}
A.~Blumberg, R.~Cohen, and C.~Teleman, \emph{Open-closed field theories, string
  topology, and {H}ochschild homology}, Alpine perspectives on algebraic
  topology, Contemp. Math., vol. 504, Amer. Math. Soc., Providence, RI, 2009,
  pp.~53--76. \MR{2581905}

\bibitem[BH04]{brunner2004}
I.~Brunner and K.~Hori, \emph{Orientifolds and mirror symmetry}, J. High Energy
  Phys. (2004), no.~11, 005, 119 pp. \MR{2118829 (2006e:81249)}

\bibitem[Bra14]{braun2014}
C.~Braun, \emph{Involutive {$A_\infty$}-algebras and dihedral cohomology}, J.
  Homotopy Relat. Struct. \textbf{9} (2014), no.~2, 317--337. \MR{3258685}

\bibitem[BS02]{braun2002}
V.~Braun and B.~Stefanski, Jr., \emph{Orientifolds and {$K$}-theory}, Cargese
  2002, Progress in String, Field and Particle Theory, NATO Science Series II:
  Mathematics, Physics and Chemistry, Springer Netherlands, 2002, pp.~369--372.

\bibitem[Cos07]{costello2007}
K.~Costello, \emph{Topological conformal field theories and {C}alabi-{Y}au
  categories}, Adv. Math. \textbf{210} (2007), no.~1, 165--214. \MR{2298823
  (2008f:14071)}

\bibitem[CW10]{caldararu2010}
A.~C{\u{a}}ld{\u{a}}raru and S.~Willerton, \emph{The {M}ukai pairing. {I}. {A}
  categorical approach}, New York J. Math. \textbf{16} (2010), 61--98.
  \MR{2657369 (2011g:18012)}

\bibitem[DFM11]{distler2011b}
J.~Distler, D.~Freed, and G.~Moore, \emph{Spin structures and superstrings},
  Surveys in differential geometry. {V}olume {XV}. {P}erspectives in
  mathematics and physics, Surv. Differ. Geom., vol.~15, Int. Press,
  Somerville, MA, 2011, pp.~99--130. \MR{2815726}

\bibitem[DGRKS07]{diaconescu2007}
D.-E. Diaconescu, A.~Garcia-Raboso, R.~Karp, and K.~Sinha, \emph{D-brane
  superpotentials in {C}alabi-{Y}au orientifolds}, Adv. Theor. Math. Phys.
  \textbf{11} (2007), no.~3, 471--516. \MR{2322534 (2009f:32046)}

\bibitem[Dij89]{dijkgraaf1989}
R.~Dijkgraaf, \emph{A geometrical approach to two-dimensional conformal field
  theory}, 1989, Thesis (Ph.D.)--Utrecht University.

\bibitem[DW90]{dijkgraaf1990}
R.~Dijkgraaf and E.~Witten, \emph{Topological gauge theories and group
  cohomology}, Comm. Math. Phys. \textbf{129} (1990), no.~2, 393--429.
  \MR{1048699}

\bibitem[Dys62]{dyson1962}
F.~Dyson, \emph{{The Threefold Way. Algebraic Structure of Symmetry Groups and
  Ensembles in Quantum Mechanics}}, Journal of Mathematical Physics \textbf{3}
  (1962), no.~6, 1199--1215.

\bibitem[FH21]{freed2021}
D.~Freed and M.~Hopkins, \emph{Consistency of {M}-theory on non-orientable
  manifolds}, Q. J. Math. \textbf{72} (2021), no.~1-2, 603--671. \MR{4271397}

\bibitem[FM13]{freed2013b}
D.~Freed and G.~Moore, \emph{Twisted equivariant matter}, Ann. Henri Poincar\'e
  \textbf{14} (2013), no.~8, 1927--2023. \MR{3119923}

\bibitem[Fre94]{freed1994}
D.~Freed, \emph{Higher algebraic structures and quantization}, Comm. Math.
  Phys. \textbf{159} (1994), no.~2, 343--398. \MR{1256993}

\bibitem[FS06]{frobenius1906}
G.~Frobenius and I.~Schur, \emph{Uber die reellen {D}arstellungen der endlichen
  {G}ruppen}, Sitzungsberichte der k\"{o}niglich preussichen Akademi der
  Wissenschaften zu Berlin (1906), 198--208.

\bibitem[GH10]{gao2011}
D.~Gao and K.~Hori, \emph{On the structure of the {C}han--{P}aton factors for
  {D}-branes in type {II} orientifolds}, ar{X}iv:1004.3972, 2010.

\bibitem[GI21]{georgieva2021}
P.~Georgieva and E.-N. Ionel, \emph{A {K}lein {TQFT}: the local real
  {G}romov-{W}itten theory of curves}, Adv. Math. \textbf{391} (2021), Paper
  No. 107972, 70. \MR{4301490}

\bibitem[Gow79]{gow1979}
R.~Gow, \emph{Real-valued and {$2$}-rational group characters}, J. Algebra
  \textbf{61} (1979), no.~2, 388--413. \MR{559848}

\bibitem[HLS21]{huang2021}
T.-C. Huang, Y.-H. Lin, and S.~Seifnashri, \emph{Construction of
  two-dimensional topological field theories with non-invertible symmetries},
  J. High Energy Phys. (2021), no.~12, 43 pp.

\bibitem[HW08]{hori2008}
K.~Hori and J.~Walcher, \emph{D-brane categories for orientifolds---the
  {L}andau-{G}inzburg case}, J. High Energy Phys. \textbf{4} (2008), 030, 36.
  \MR{2425273 (2010d:81251)}

\bibitem[IT23]{ichikawa2023}
T.~Ichikawa and Y.~Tachikawa, \emph{The super {F}robenius--{S}chur indicator
  and finite group gauge theories on $\textnormal{Pin}^-$ surfaces}, Comm.
  Math. Phys. \textbf{400} (2023), no.~1, 417--428.

\bibitem[Kar70]{karoubi1970}
M.~Karoubi, \emph{Sur la {$K$}-th\'eorie \'equivariante}, S\'eminaire
  {H}eidelberg-{S}aarbr\"ucken-{S}trasbourg sur la {K}-th\'eorie (1967/68),
  Lecture Notes in Mathematics, Vol. 136, Springer, Berlin, 1970, pp.~187--253.
  \MR{0268253}

\bibitem[Kar85]{karpilovsky1985}
G.~Karpilovsky, \emph{Projective representations of finite groups}, Monographs
  and Textbooks in Pure and Applied Mathematics, vol.~94, Marcel Dekker, Inc.,
  New York, 1985. \MR{788161}

\bibitem[Kho11]{khoi2011}
V.~Khoi, \emph{On {T}uraev's theorem about {D}ijkgraaf-{W}itten invariants of
  surfaces}, J. Knot Theory Ramifications \textbf{20} (2011), no.~6, 837--846.
  \MR{2812266}

\bibitem[KM97]{karimipour1997}
V.~Karimipour and A.~Mostafazadeh, \emph{Lattice topological field theory on
  nonorientable surfaces}, J. Math. Phys. \textbf{38} (1997), no.~1, 49--66.
  \MR{1424641 (99a:58026)}

\bibitem[KR04]{kapustin2004}
A.~Kapustin and L.~Rozansky, \emph{On the relation between open and closed
  topological strings}, Comm. Math. Phys. \textbf{252} (2004), no.~1-3,
  393--414. \MR{2104884 (2006a:81134)}

\bibitem[KT17]{kapustin2017}
A.~Kapustin and A.~Turzillo, \emph{Equivariant topological quantum field theory
  and symmetry protected topological phases}, J. High Energy Phys. (2017),
  no.~3, 006, front matter+19. \MR{3657691}

\bibitem[Laz01]{lazaroiu2001}
C.~Lazaroiu, \emph{On the structure of open-closed topological field theory in
  two dimensions}, Nuclear Phys. B \textbf{603} (2001), no.~3, 497--530.
  \MR{1839382 (2002h:81238)}

\bibitem[LN11]{loktev2011}
S.~Loktev and S.~Natanzon, \emph{Klein topological field theories from group
  representations}, SIGMA Symmetry Integrability Geom. Methods Appl. \textbf{7}
  (2011), Paper 070, 15. \MR{2861206}

\bibitem[LP08]{lauda2008}
A.~Lauda and H.~Pfeiffer, \emph{Open-closed strings: two-dimensional extended
  {TQFT}s and {F}robenius algebras}, Topology Appl. \textbf{155} (2008), no.~7,
  623--666. \MR{2395583}

\bibitem[Med78]{mednykh1978}
A.~Mednykh, \emph{Determination of the number of nonequivalent coverings over a
  compact {R}iemann surface}, Dokl. Akad. Nauk SSSR \textbf{239} (1978), no.~2,
  269--271. \MR{490616}

\bibitem[MS06]{moore2006}
G.~Moore and G.~Segal, \emph{D-branes and {$K$}-theory in {2D} topological
  field theory}, ar{X}iv:hep-th/0609042, 2006.

\bibitem[MY05]{mulase2005}
M.~Mulase and J.~Yu, \emph{Non-commutative matrix integrals and representation
  varieties of surface groups in a finite group}, Ann. Inst. Fourier (Grenoble)
  \textbf{55} (2005), no.~6, 2161--2196. \MR{2187951}

\bibitem[NY22]{noohiYoung2022}
B.~Noohi and M.~Young, \emph{Twisted loop transgression and higher {J}andl
  gerbes over finite groupoids}, Algebr. Geom. Topol. \textbf{22} (2022),
  no.~4, 1663--1712.

\bibitem[RT21]{rumynin2021}
D.~Rumynin and J.~Taylor, \emph{Real representations of {$C_2$}-graded groups:
  the antilinear theory}, Linear Algebra Appl. \textbf{610} (2021), 135--168.
  \MR{4159287}

\bibitem[RT22]{rumynin2021b}
\bysame, \emph{Real representations of {$C_2$}-graded groups: the linear and
  hermitian theories}, High. Struct. \textbf{6} (2022), no.~1, 359--374.
  \MR{4456598}

\bibitem[RY21]{rumyninYoung2021}
D.~Rumynin and M.~Young, \emph{Burnside rings for {R}eal 2-representation
  theory: {T}he linear theory}, Commun. Contemp. Math. \textbf{23} (2021),
  no.~5, 2050012, 54. \MR{4289902}

\bibitem[Sha11]{sharpe2011}
E.~Sharpe, \emph{Notes on discrete torsion in orientifolds}, J. Geom. Phys.
  \textbf{61} (2011), no.~6, 1017--1032. \MR{2782477}

\bibitem[Sha15]{sharpe2015}
\bysame, \emph{Notes on generalized global symmetries in {QFT}}, Fortschr.
  Phys. \textbf{63} (2015), no.~11-12, 659--682. \MR{3422349}

\bibitem[Shi12]{shimizu2012}
K.~Shimizu, \emph{Frobenius–{S}chur indicator for categories with duality},
  Axioms \textbf{1} (2012), no.~3, 324--364.

\bibitem[Sny17]{snyder2017}
N.~Snyder, \emph{Mednykh's formula via lattice topological quantum field
  theories}, Proceedings of the 2014 {M}aui and 2015 {Q}inhuangdao conferences
  in honour of {V}aughan {F}. {R}. {J}ones' 60th birthday, Proc. Centre Math.
  Appl. Austral. Nat. Univ., vol.~46, Austral. Nat. Univ., Canberra, 2017,
  pp.~389--398. \MR{3635678}

\bibitem[SR17]{shiozaki2017}
K.~Shiozaki and S.~Ryu, \emph{Matrix product states and equivariant topological
  field theories for bosonic symmetry-protected topological phases in {$(1+1)$}
  dimensions}, J. High Energy Phys. (2017), no.~4, 100, front matter+46.
  \MR{3650168}

\bibitem[TT06]{turaev2006}
V.~Turaev and P.~Turner, \emph{Unoriented topological quantum field theory and
  link homology}, Algebr. Geom. Topol. \textbf{6} (2006), 1069--1093.
  \MR{2253441 (2007f:57061)}

\bibitem[Tur07]{turaev2007}
V.~Turaev, \emph{Dijkgraaf-{W}itten invariants of surfaces and projective
  representations of groups}, J. Geom. Phys. \textbf{57} (2007), no.~11,
  2419--2430. \MR{2360250 (2009b:57063)}

\bibitem[Wig59]{wigner1959}
E.~Wigner, \emph{Group theory and its application to the quantum mechanics of
  atomic spectra}, Pure and Applied Physics. Vol. 5, Academic Press, New
  York-London, 1959. \MR{0106711}

\bibitem[Wil08]{willerton2008}
S.~Willerton, \emph{The twisted {D}rinfeld double of a finite group via gerbes
  and finite groupoids}, Algebr. Geom. Topol. \textbf{8} (2008), no.~3,
  1419--1457. \MR{2443249}

\bibitem[You20]{mbyoung2020}
M.~Young, \emph{Orientation twisted homotopy field theories and twisted
  unoriented {D}ijkgraaf--{W}itten theory}, Comm. Math. Phys. \textbf{374}
  (2020), no.~3, 1645--1691.

\bibitem[You21]{mbyoung2021b}
\bysame, \emph{Real representation theory of finite categorical groups}, High.
  Struct. \textbf{5} (2021), no.~1, 18--70.

\end{thebibliography}

\newcommand{\etalchar}[1]{$^{#1}$}
\providecommand{\bysame}{\leavevmode\hbox to3em{\hrulefill}\thinspace}
\providecommand{\MR}{\relax\ifhmode\unskip\space\fi MR }
\providecommand{\MRhref}[2]{%
  \href{http://www.ams.org/mathscinet-getitem?mr=#1}{#2}
}
\providecommand{\href}[2]{#2}

\end{document}